\documentclass[onefignum,onetabnum]{siamart190516}  

\usepackage{amsfonts}
\usepackage{amsopn}
\usepackage{amsmath}
\usepackage{amssymb}
\usepackage{multirow}
\usepackage{tabularx}
\usepackage{hhline}
\usepackage{scalerel,stackengine}
\usepackage{arydshln}
\usepackage[normalem]{ulem}
\usepackage{soul}
\usepackage{tikz}
\usepackage{pgf}
\usepackage{collcell}
\usepackage[utf8]{inputenc}
\usepackage{geometry}
\usepackage{multicol}
\usepackage{graphbox}
\usepackage{setspace}
\usepackage{xspace}

\newcolumntype{E}{>{\collectcell\usermacro}c<{\endcollectcell}}
\newcommand\usermacro[1]{\pgfmathparse{10000*#1}\pgfmathprintnumber\pgfmathresult}

\usepackage{algorithmicx}
\usepackage{algorithm}
\usepackage{algpseudocode}

\makeatletter
\patchcmd{\ALG@step}{\addtocounter{ALG@line}{1}}{\refstepcounter{ALG@line}}{}{}
\newcommand{\ALG@lineautorefname}{Step}
\makeatother

\usepackage{cite}
\usepackage[shortlabels]{enumitem}

\DeclareMathAlphabet{\mathup}{OT1}{\familydefault}{m}{n}

\DeclareSymbolFont{yhlargesymbols}{OMX}{yhex}{m}{n}
\DeclareMathAccent{\wideparen}{\mathord}{yhlargesymbols}{"F3}
\xdefinecolor{green}{rgb}{0.04, 0.85, 0.32}
\xdefinecolor{darkgreen}{rgb}{0.24, 0.7, 0.44}
\xdefinecolor{mint}{rgb}{0.24, 0.71, 0.54}
\xdefinecolor{officegreen}{rgb}{0.0, 0.5, 0.0}
\xdefinecolor{napiergreen}{rgb}{0.16, 0.5, 0.0}
\xdefinecolor{maroon}{rgb}{0.65,0.06,0.17}
\xdefinecolor{blue}{rgb}{0,0.2,0.6}
\xdefinecolor{phthalogreen}{rgb}{0.07,0.21,0.14}
\definecolor{grassgreen}{RGB}{92,135,39}

\crefname{algocf}{alg.}{algs.}
\Crefname{algocf}{Algorithm}{Algorithms}

\newcommand{\logdet}[1]{\log\det\left(#1\right)}                    
\newcommand{\diag}[1]{\mathsf{diag}\left(#1\right)}                 
\newcommand{\del}[2]{\frac{\partial{#1}}{\partial{#2}}}             
\newcommand{\mat}[1]{\mathbf{{#1}}}                                 
\renewcommand{\vec}[1]{\mathbf{{#1}}}                                 

\DeclareMathOperator*{\argmin}{arg\,min}  
\DeclareMathOperator*{\argmax}{arg\,max}  

\newcommand{\trunc}{\mat{P}}       
\newcommand{\Trunc}[2]{\trunc_{#1}{\left(#2\right)}}       
\newcommand{\proj}{\mat{L}}       

\newcommand{\regpenalty}{\alpha}                    
\newcommand{\penaltyfunc}{\Phi}                    
\newcommand{\penaltyfunction}[1]{\penaltyfunc({#1})}                    
\newcommand*{\tran}{^{\mkern-1.5mu\mathsf{T}}}                
\newcommand*{\adj}{^{\mkern-1.5mu\mathsf{*}}}                 
\newcommand*{\inv}{^{\mkern-1.5mu\mathsf{-1}}}                
\newcommand{\inverse}[1]{\Bigl(#1\Bigr)\inv}                  
\newcommand{\pseudoinverse}[1]{\Bigl(#1\Bigr)^{\dagger}}          
\newcommand{\domain}{\mathcal{D}}                             
\newcommand{\trace}{\mathrm{Tr}}                              
\newcommand{\Trace}[1]{\trace \left(#1\right)}                
\newcommand{\abs}[1]{\left| {#1} \right|}                  
\newcommand{\wnorm}[2]{\left\| {#1} \right\|_{#2}}            
\newcommand{\sqwnorm}[2]{\left\| {#1} \right\|^2_{#2}}        

\newcommand\restr[2]{{ \left.\kern-\nulldelimiterspace        
                     {#1}\vphantom{\big|} \right|_{#2}}}

\newcommand{\Rnum}{\mathbb{R}}  
\newcommand{\xcont}{u}                             

\newcommand{\y}{\mathbf{y}}                        
\newcommand{\obs}{\y}                              

\newcommand{\param}{\vec{\theta}}                  
\newcommand{\iparam}{\param}                       
\newcommand{\iparprior}{\iparam_{\rm pr}}          
\newcommand{\iparb}{\iparprior}                    
\newcommand{\iparpost}{\iparam_{\rm post}^\obs}    
\newcommand{\ipara}{\iparpost}                     

\newcommand{\ipartrue}{\iparam_{\rm true}}
\newcommand{\Nobs}{\textsc{N}_{\rm obs}}                        
\newcommand{\Nens}{\textsc{N}_{\rm ens}}                        
\newcommand{\nobs}{n_{t}}                                  
\newcommand{\nobstimes}{\nobs}                             
\newcommand{\Nsens}{{n_{\rm s}}}                         
\newcommand{\budget}{\tau}                         

\newcommand{\ncr}[2]{C(#1, #2)}
\newcommand{\Cparamprior}{\mat\Gamma_{{\rm pr}}}                
\newcommand{\Cparampost}{\mat\Gamma_{{\rm post}}}               
\newcommand{\Cobsnoise}{\mat{\Gamma}_{ {\rm noise}}}            
\newcommand{\Cobsnoiseinv}{\mat{\Gamma}\inv_{ {\rm noise}}}            
\newcommand{\Cparampriormat}{\Cparamprior}                      
\newcommand{\Cparampostmat}{\Cparampost}                        

\newcommand{\Fcont}{\mathcal{F}}                                 
\newcommand{\F}{\mathbf{F}}                                      

\newcommand{\Fadj}{\F\adj}

\newcommand{\obj}{\mathcal{J}}            
\newcommand{\stochobj}{\Upsilon}          

\newcommand{\FIM}{\mathsf{FIM}}
\newcommand{\utilityfunc}{\mathcal{U}}

\newcommand{\uncertainparam}{\llam}
\newcommand{\uncertainparamset}{\Lambda}
\newcommand{\uncertainparamsample}{\widehat{\Lambda}}
\newcommand{\WCopt}{^{\mathup{WC-opt}}}  

\newcommand{\func}{f}         

\newcommand{\Prob}{\mathbb{P}}                                   
\newcommand{\CondProb}[2]{\mathbb{P}\left(#1|#2 \right)}  
\newcommand{\GM}[2]{\mathcal{N}\!\left( {#1}, {#2}\right)}       

\newcommand{\like}{\mathcal{L}}                          
\newcommand{\Like}[2]{\like{\left(#1|#2\right)}}                                  

\newcommand{\Expect}[2]{\mathbb{E}_{#1}{\Bigl[ #2 \Bigr]} }     
\newcommand{\baseline}{b}
\newcommand{\hyperparam}{\vec{p}}     
\newcommand{\design}{\boldsymbol{\zeta}}                                       
                                       
\newcommand{\weightfunc}{\omega}                                    
\newcommand{\varweightfunc}{\varpi}    
\newcommand{\designmat}{\mat{W}}                                    
\newcommand{\wdesignmat}{\designmat_{\Gamma}}                       

\newcommand{\llam}{ {\boldsymbol\lambda} }

\stackMath
\newcommand\reallywidehat[1]{%
\savestack{\tmpbox}{\stretchto{%
  \scaleto{%
    \scalerel*[\widthof{\ensuremath{#1}}]{\kern-.6pt\bigwedge\kern-.6pt}%
    {\rule[-\textheight/2]{1ex}{\textheight}}
  }{\textheight}%
}{0.5ex}}%
\stackon[1pt]{#1}{\tmpbox}%
}
\newcommand{\pyoed}{{PyOED}\xspace}

\newcommand*{\opt}{^{\mkern-1.5mu\mathsf{opt}}}               

\newcommand{\commentout}[1]{\iffalse {#1} \fi}

\newcommand{\sven}[1]{\textcolor{blue}{Sven: #1}\marginpar{\textcolor{blue}{SL}}}

\newcommand{\ahmed}[1]{\textcolor{officegreen}{Ahmed:
#1}\marginpar{\textcolor{officegreen}{AA}}}

\ifpdf
  \DeclareGraphicsExtensions{.eps,.pdf,.png,.jpg}
\else
  \DeclareGraphicsExtensions{.eps}
\fi

\newsiamremark{remark}{Remark}
\newsiamremark{hypothesis}{Hypothesis}
\crefname{hypothesis}{Hypothesis}{Hypotheses}
\newsiamthm{claim}{Claim}
\crefname{lemma}{Lemma}{Lemmas}

\patchcmd{\SetTagPlusEndMark}{$}{}{}{}
\patchcmd{\SetTagPlusEndMark}{$}{}{}{}

\headers{Robust A-Opt OED}{A. Attia, S. Leyffer,  and  T.  Munson}

\title{Robust A-Optimal Experimental Design for Bayesian Inverse Problems\thanks{Submitted to the editors \today.
\funding{
  This material is based upon work supported in part by the U.S. Department of Energy, Office of Science, Office of Advanced Scientific Computing Research, Scientific Discovery through Advanced Computing (SciDAC) Program through the FASTMath Institute under contract number DE-AC02-06CH11357 at Argonne National Laboratory.}}
}

\author{Ahmed Attia\thanks{Mathematics and Computer Science Division,
                   Argonne National Laboratory, USA
                  ( \email{attia@mcs.anl.gov} ).}
\and Sven Leyffer\thanks{Mathematics and Computer Science Division,
                   Argonne National Laboratory, USA
                   ( \email{leyffer@anl.gov} ).}
\and Todd Munson\thanks{Mathematics and Computer Science Division,
                   Argonne National Laboratory, USA
                   ( \email{tmunson@mcs.anl.gov} ).}
}

\newif\ifshowdetails
\newif\ifshowextensivedetails
\newif\ifshowresults
\newif\ifshowalgorithms
\newif\ifshowappendix
\newif\ifshowall

\showalltrue

\ifshowall
  \showdetailstrue
  \showextensivedetailstrue
  \showresultstrue
  \showalgorithmstrue
  \showappendixtrue
\else
  \showdetailsfalse
  \showextensivedetailsfalse
  \showresultstrue
  \showalgorithmsfalse
  \showappendixtrue
\fi

\ifpdf
\hypersetup{
  pdftitle={Robust A-Optimal Design for Bayesian Inversion},
  pdfauthor={A. Attia, S. Leyffer, and T. Munson}
}
\fi

\begin{document}


\maketitle

\begin{abstract}
  Optimal design of experiments for Bayesian inverse problems has recently
  gained wide popularity and attracted much attention, 
  especially in the computational science and Bayesian inversion communities.
  An optimal design maximizes a predefined utility
  function that is formulated in terms of the elements of an inverse problem,
  an example being optimal sensor placement for parameter identification.
  The state-of-the-art algorithmic approaches following this simple formulation 
  generally overlook misspecification of the elements of the inverse problem, 
  such as the prior or the measurement uncertainties.
  This work presents an efficient algorithmic approach for designing 
  optimal experimental design schemes for Bayesian inverse problems such that the optimal design is
  robust to misspecification of elements of the inverse problem.
  Specifically, we consider a worst-case scenario approach for the uncertain
  or misspecified parameters, formulate robust objectives, and propose
  an algorithmic approach for optimizing such objectives. Both relaxation and
  stochastic solution approaches are discussed with detailed analysis and insight into 
  the interpretation of the problem and the proposed algorithmic approach. 
  Extensive numerical experiments to validate and analyze the proposed approach
  are carried out for sensor placement in a parameter identification problem.
\end{abstract}





\section{Introduction} 
  \label{sec:introduction}
  Inverse problems play a fundamental rule in simulation-based
  prediction and are instrumental in scientific discovery.
  Data assimilation provides a set of tools tailored for solving inverse
  problems in large- to extreme-scale settings and producing accurate, trustworthy
  predictions; see, 
  for example,\cite{bannister2017review,daley1993atmospheric,navon2009data,attia2016reducedhmcsmoother,attia2015hmcfilter}.
  The quality of the solution of an inverse problem is greatly influenced by the
  quality of the observational configurations and the collected data.
  Optimal experimental design (OED) provides the tools for optimally designing
  observational configurations and data acquisition
  strategies~\cite{cox1958planning,Pukelsheim93,FedorovLee00,Ucinski00,Pazman86} in order to optimize 
  a predefined objective, for example, to maximize
  the information gain from an experiment or observational data.
  OED has seen a recent surge of interest in the computational science and
  Bayesian inversion communities; see, for example,~\cite{HaberHoreshTenorio08,HaberMagnantLuceroEtAl12,HaberHoreshTenorio10,HuanMarzouk13,bui2013computational,alexanderian2016bayesian,AlexanderianSaibaba17,AlexanderianPetraStadlerEtAl16,AlexanderianPetraStadlerEtAl14,alexanderian2021optimal,attia2018goal,attia2022optimal}.

  Bayesian OED approaches define the optimal design as the one that
  maximizes the information gain from the observation and/or experimental 
  data in order to maximize the quality of the
  inversion parameter, that is, the solution of the Bayesian inverse problem. 
  In a wide range of scientific applications, however, the elements of an inverse problem can be
  misspecified or be loosely described.
  Thus, the optimal design needs to account for uncertainty or misspecification
  of the elements of the Bayesian inverse problem.
  Generally speaking, robust OED~\cite{asprey2002designing,zang2005review,telen2012robust,rojas2007robust,telen2014robustifying,pronzato1985robust,pronzato1988robust,pronzato2004minimax} 
  is concerned with optimizing observational configurations against average or 
  worst-case scenarios resulting from an uncertain or misspecified parameter.
  The former, however, requires associating a probabilistic description to the
  uncertain parameter and results in a design that is optimal on average.
  Conversely, the latter seeks a design that is optimal against a lower bound of
  the objective over an admissible set of the uncertain parameter and thus
  does not require a probabilistic description and is optimal against
  the worst-case scenario. 
  This is achieved by the so-called max-min
  design~\cite{pronzato1988robust,biedermann2003note,dette2003standardized}, which has been given
  little attention in model-constrained Bayesian OED. 
  

  This work is concerned with developing a formal approach and algorithm for
  solving robust OED in Bayesian inversion. 
  Specifically, we develop an algorithmic approach for solving sensor placement
  problems in Bayesian inversion by finding a max-min optimal design where the 
  prior and/or observation uncertainties/covariances are misspecified. 
  This robustness is important especially when little or no sufficient information is available a priori about the inversion parameter or when the information about
  sensor accuracy is not perfectly specified. 
  Moreover, this can help ameliorate the effect of misspecification of representativeness
  errors resulting from inaccurate observational operators.
  The methodology described in this work can be easily extended to account for other
  uncertain or misspecified parameters, such as the prior mean, the length of the simulation window, 
  and the temporal observation frequency.
  We focus on A-optimal designs for linear/linearized Bayesian
  OED problems. Extensions to other utility functions, however, can be easily
  developed by following the strategy presented in this work.


  The paper is organized as follows. 
  \Cref{sec:background} provides the mathematical background for Bayesian inversion, OED, and robust design.
  In~\Cref{sec:robust_oed} we present approaches for formulating and solving robust OED problems for Bayesian inversion. 
  Numerical results are given in~\Cref{sec:numerical_results}. 
  Concluding remarks are given in~\Cref{sec:conclusions}.

\section{Background}
\label{sec:background}
  In this section we provide a brief overview of Bayesian inversion 
  and Bayesian OED problems for optimal data acquisition and sensor placement.
  
  \subsection{Bayesian inversion}
    \label{subsec:background_Bayesian_inversion}
    Consider the forward model described by
    \begin{equation}\label{eqn:forward_problem}
      \obs  = \Fcont(\iparam) + \vec{\delta} \,,
    \end{equation}
    %
    where $\iparam$ is the model parameter, 
    $\obs \in \Rnum^{\Nobs}$ is the observation and
    $\vec{\delta} \in \Rnum^{\Nobs}$ is the observation error.
    In most applications, observational errors---modeling the mismatch between model 
    predictions and actual data---are assumed to be Gaussian 
    $\vec{\delta} \sim \GM{\vec{0}}{\Cobsnoise}$,
    where $\Cobsnoise$ is the observation error covariance matrix.
    In this case, the data likelihood is 
    \begin{equation} \label{eqn:Gaussian_likelihood}
      \Like{\obs}{ \iparam } \propto
        \exp{\left( - \frac{1}{2} 
          \sqwnorm{ \Fcont(\iparam) - \obs }{ \Cobsnoise\inv } \right) } \,,
    \end{equation}
    where the matrix-weighted norm is defined as
    $\sqwnorm{\vec{x}}{\mat{A}} = \vec{x}\tran \mat{A} \vec{x} $ for a vector 
    $\vec{x}$ and a matrix $\mat{A}$ of conformable dimensions.
    Note, however, that the approach presented in this work (see Section~\ref{sec:robust_oed}) 
    is not limited to Gaussian observation errors and can be extended to other error models. 
    
    An inverse problem refers to the retrieval of the model parameter $\iparam$ from
    noisy observations $\obs$, conditioned by the model dynamics.
    In Bayesian inversion, a prior encoding any current knowledge about the parameter is specified, and 
    the goal is to study the probability distribution of
    $\iparam$ conditioned by the observation $\obs$, that is, the posterior
    obtained by applying Bayes' theorem.
    By assuming a Gaussian prior~$\iparam \sim \GM{\iparb}{\Cparampriormat}$
    and a linear parameter-to-observable map $\Fcont=\F$, the posterior is 
    also Gaussian $\GM{\ipara}{\Cparampostmat}$ with
    \begin{equation}\label{eqn:Posterior_Params}
      \Cparampostmat = \left(\Fcont \adj \Cobsnoise\inv \Fcont 
        + \Cparampriormat\inv \right)\inv \,, \quad
      \ipara = \Cparampostmat \left( \Cparampriormat\inv \iparb 
        + \Fcont\adj \Cobsnoise\inv \, \obs \right) \,.
    \end{equation}

    This setup, known as the linear Gaussian case, is popular and is used in
    many applications for simplicity, especially in data assimilation algorithms 
    used for solving large-scale inverse problems for numerical weather prediction. 
    Tremendous effort is in progress to account for the high nonlinearity of the
    simulations and observations and to handle non-Gaussian errors, such as
    work on sampling-based methods including Markov chain Monte Carlo~\cite{attia2015hmcfilter,attia2015hmcsampling,attia2015hmcsmoother,attia2017reduced,attia2018ClHMCAtmos} and
    particle filtering~\cite{navon2009data,vetra2018state}.

  \subsection{OED for Bayesian inversion}
    \label{subsec:background_oed}
    A binary OED optimization problem takes the general form
    \begin{equation}\label{eqn:binary_OED_optimization_unregularized}
      \design\opt
      = \argmax_{\design \in \{0, 1\}^{\Nsens}} \,
        \utilityfunc(\design) \,,
    \end{equation}
    where $\design\in \{0,1\}^{\Nsens}$ is a binary design variable 
    ($1$ means an active candidate, and $0$ corresponds to an inactive one) 
    associated with candidate configurations such as sensor locations or types and 
    $\utilityfunc(\design)$ is a predefined utility function
    that quantifies the design quality associated with an inference
    parameter/state of an inverse problem; see~\Cref{subsec:background_utility_functions}.
    In Bayesian inversion, candidates can correspond to proposed sensor locations, 
    observation time points, or other control variables that have influence on the quality of the solution of the
    inverse problem or the predictive power of the assimilation system. 
    The utility function can, for example, be set to a function that summarizes the information 
    gain from the observational data.
    We focus the discussion hereafter on sensor placement as an application for clarity of the presentation. 
    We note, however, that the work presented here can be extended to other binary control variables and utility functions.
    
    The optimization problem~\eqref{eqn:binary_OED_optimization_unregularized} is often appended with a
    sparsity enforcing term $ - \regpenalty\, \penaltyfunc(\design);\,\regpenalty \geq 0 $ to prevent dense
    designs---the function $\penaltyfunc(\design)$ asserts regularization or sparsity on
    the design---and thus reduces the cost associated with deploying observational sensors. 
    The regularized binary OED optimization problem takes the form 
    \begin{equation}\label{eqn:binary_OED_optimization}
      \design\opt
      = \argmax_{\design \in \{0, 1\}^{\Nsens}} \,
        \obj(\design):= \utilityfunc(\design) - \regpenalty \penaltyfunc(\design) \,.
    \end{equation}

    The penalty (soft constraint) could, for example,  encode an observational resource constraint:
    $\sum_{i}^{\Nsens}\design_i \leq k $ or
    $\sum_{i}^{\Nsens}\design_i = k \,; k\in \mathtt{Z}_{+}\,,$
    for example,  an upper bound (or exact budget) on the number of sensors to be deployed.
    This penalty function can also be a sparsifying (possibly nondifferentiable) function, for example,
    $\wnorm{\design}{0}$.
    Note that $\penaltyfunc(\design)$ is an auxiliary term added after formulating the optimization problem 
    and that the sign of the penalty term is negative.
    We define the OED optimization problem~\eqref{eqn:binary_OED_optimization} as a maximization 
    because we use the term ``utility function.''
    Note, however, the maximization can be replaced easily with an equivalent minimization problem where 
    the utility function is replaced with an equivalent optimality criterion and the sign of the penalty 
    term is flipped, that is, set to positive. 
    See~\Cref{subsec:background_utility_functions} for further details.
    \commentout{
      We could also add constraints directly,
      making the subproblem harder to solve, but potentially allowing us to solve more
      interesting problems.
      However, adding constraints such as simplex constraints, or maximum number of
      sensors will require redefining the feasibility set. We will also need to
      project the gradient onto that feasibility set. This might be tricky if the
      feasibility set of the uncertain parameter is non-trivial.
      AHMED: Another approach to accommodate hard constraints on the design is to associate an indicator function
      to the binary design, thus modifying the Bernoulli probabilities.
      This is interesting, and can be a good starting point for extensions of the stochastic learning approach in general.
    }

    Solving~\eqref{eqn:binary_OED_optimization} using traditional binary
    optimization approaches is computationally prohibitive and is typically
    replaced with a relaxation
    \begin{equation}\label{eqn:relaxed_OED_optimization}
      \design\opt
      = \argmax_{\design \in [0, 1]^{\Nsens}} \,
        \obj(\design):= \utilityfunc(\design) - \regpenalty \penaltyfunc(\design) \,,
    \end{equation}
    where the design variable associated with any sensor is relaxed to take
    any value in the interval $[0, 1]$ and then a gradient-based optimization approach is 
    used to solve the relaxed optimization problem~\eqref{eqn:relaxed_OED_optimization} 
    for an estimate of the solution of the 
    original binary optimization problem~\eqref{eqn:binary_OED_optimization}. 
    In many cases, this approach requires applying a rounding methodology, such as sum-up rounding to the solution
    of~\eqref{eqn:relaxed_OED_optimization} to obtain an estimate of the solution of
    the original binary optimization problem~\cite{YuZavalaAnitescu17}.
    %
    Little attention, however, is given to the effect of this relaxation on the OED
    optimization problem and on the resulting experimental design.
    In general, relaxation produces a surface that connects the corners of the binary domain
    (or approximation thereof) on which a gradient-based optimization
    algorithm is applied to solve~\eqref{eqn:relaxed_OED_optimization}.
    Thus, when applying relaxation, one must assure that the relaxed objective
    $\obj$ is continuous over the domain $[0, 1]^{\Nsens}$ with proper limits at
    the bounds/corners of the domain. This issue has been extensively discussed
    in~\cite{attia2022optimal} and will be revisited in~\Cref{subsec:background_utility_functions}.

    Recently, a stochastic learning approach to binary OED was presented
    in~\cite{attia2022stochastic} to efficiently solve~\eqref{eqn:binary_OED_optimization}, 
    where the optimal design is obtained by sampling the optimal parameters:
    \begin{equation}\label{eqn:stochastic_OED_optimization}
      \hyperparam\opt
        = \argmax_{\hyperparam\in[0,1]^{\Nsens}}
          \stochobj(\hyperparam):=
            \Expect{\design\sim\CondProb{\design}{\hyperparam}}{\obj(\design)}\, 
          \,,
    \end{equation}
    where $\CondProb{\design}{\hyperparam}$ is a multivariate Bernoulli
    distribution with parameter $\hyperparam$ specifying probabilities of
    success/activation of each entry of $\design$, that is, $\hyperparam_i\in[0, 1]$.
    The expectation in~\eqref{eqn:stochastic_OED_optimization} is defined 
    as 
    \begin{equation}\label{eqn:stochastic_objective}
      \Expect{\design\sim\CondProb{\design}{\hyperparam}}{\obj(\design)}
        = \sum_{k=1}^{2^{\Nsens}}{\obj(\design[k])\, \Prob{(\design[k])}}\,; \quad
          \design\in\{0,1\}^{\Nsens} \,,
    \end{equation}
    where each possible binary (vector) value of the design $\design$ is identified by a unique 
    index $k\in{1,\ldots,2^{\Nsens}}$.
    A simple indexing scheme is to define a unique index $k$ based on the elementwise values of the design $\design$ as 
    $
      k := 1 + \sum_{i=1}^{\Nsens}{\design_i\, 2 ^{i-1}} \,, \quad \design_i\in\{0,1\} \,.
    $ 
    The expectation~\eqref{eqn:stochastic_objective} is not 
    evaluated explicitly, and its derivative is efficiently 
    approximated following a stochastic gradient approach. 
    As outlined in~\cite{attia2022stochastic}, unlike~\eqref{eqn:relaxed_OED_optimization},
    the stochastic approach~\eqref{eqn:stochastic_OED_optimization} is
    computationally efficient and does not require differentiability of the
    objective $\obj(\design)$ with respect to the design.
    In this work we utilize and discuss both approaches, that is, the standard
    relaxation~\eqref{eqn:relaxed_OED_optimization} and the stochastic
    approach~\eqref{eqn:stochastic_OED_optimization}, for solving the binary OED
    optimization problem~\eqref{eqn:binary_OED_optimization} in the context of
    robust design.
    In~\Cref{subsec:background_utility_functions}, we discuss general approaches for 
    formulating the utility function $\utilityfunc$ in Bayesian OED.

  \subsection{OED utility functions}
    \label{subsec:background_utility_functions}
    The role of the binary design in sensor placement applications is to reconfigure 
    the observational vector or, equivalently, the observation error covariance matrix.
    In Bayesian OED (see, e.g., \cite{alexanderian2021optimal,alexanderian2016bayesian,AlexanderianPetraStadlerEtAl14,attia2018goal,HaberHoreshTenorio08,HaberHoreshTenorio10}), 
    by introducing the design $\design$ to the inverse problem, the observation covariance $\Cobsnoise$ is replaced 
    with a weighted version $\wdesignmat(\design)$, resulting in the weighted data-likelihood
    \begin{equation}\label{eqn:weighted_joint_likelihood}
      \Like{\obs}{ \param; \design}
        \propto \exp{\left( - \frac{1}{2} 
        \sqwnorm{ \Fcont(\param) - \obs}{\wdesignmat(\design)} \right) } \,,
    \end{equation}
    where $\wdesignmat(\design)$ is a weighted version of the precision matrix
    $\Cobsnoiseinv$ obtained by removing rows (and columns) from $\Cobsnoise$ corresponding to
    inactive sensors.
    For a binary design $\design\in\{0,1\}^{\Nsens}$, that is, for
    both~\eqref{eqn:binary_OED_optimization}
    and~\eqref{eqn:stochastic_OED_optimization},
    this can be achieved by defining the weighted precision matrix as
    \begin{equation}\label{eqn:PrePost_weighted_Precision}
      \wdesignmat(\design) 
        := \pseudoinverse{ \diag{\design} \Cobsnoise \diag{\design} } 
        \equiv
          \proj\tran \left( 
              \proj \Bigl(\diag{\design} \Cobsnoise \diag{\design} \Bigr) \proj\tran
            \right)\inv \proj
          \,, 
    \end{equation}
    where $\dagger$ is the Moore--Penrose pseudoinverse and
    $\proj:=\proj(\design) \in \Rnum^{n_a \times \Nsens}$ is a sparse matrix that extracts 
    the $n_a:=\sum \design$ nonzero rows/columns from the design matrix $\wdesignmat$, corresponding to 
    the nonzero entries of the design $\design$.
    Specifically, $\proj$ is a binary matrix with only one entry equal to $1$ on
    each row $i$, at the column corresponding to the $i$th nonzero entry in $\design$.
    In the case of relaxation, however, 
    the weighted likelihood is formulated by replacing the precision matrix
    $\Cobsnoise$ with the more general form 
    \begin{subequations}\label{eqn:pointwise_weighted_precision}
    \begin{equation}\label{eqn:Shur_weighted_Precision}
      \wdesignmat(\design) 
        := \pseudoinverse{ \designmat(\design) \odot \Cobsnoise } 
        \equiv
          \proj\tran \left( 
              \proj \Bigl( \designmat(\design) \odot \Cobsnoise \Bigr) \proj\tran
            \right)\inv \proj
          \,, 
    \end{equation}
    where $\odot$ is the Hadamard (Schur) product and $\designmat(\design)$ is
    a weighting matrix with entries defined by
    \begin{equation}\label{eqn:design_weights}
      \varweightfunc\left(\design_i, \design_j \right)
        = 
        \begin{cases}
          \weightfunc_i \weightfunc_j &;  i \neq j \\
          \begin{cases}
            0     &;  \weightfunc_i = 0 \\
            \frac{1}{\weightfunc_i ^ 2}  &; \weightfunc_i \neq 0  
          \end{cases} &; i=j \\ 
        \end{cases} ; i,j = 1,2,\ldots,\Nsens \,,  
    \end{equation}
    \end{subequations}
    where $\weightfunc_i=\weightfunc(\design_i)\in[0, 1]$ is a weight 
    associated with the $i$th candidate sensor and is calculated based on the value of 
    the relaxed design $\design_i \in [0, 1]$.
    The simplest form for this function is $\weightfunc_i(\design_i):=\design_i$; 
    see~\cite{attia2022optimal} for further details
    and for other forms of the weighting functions $\varweightfunc$.
    This formulation~\eqref{eqn:pointwise_weighted_precision} of the weighted precision matrix
    preserves the continuity of the OED objective at the bounds of the domain and
    thus produces a relaxation surface that connects the values of the
    OED objective function evaluated at the corners of the domain, that is, at the binary
    designs.
    Note that the general form of the weighted precision
    matrix~\eqref{eqn:pointwise_weighted_precision} reduces 
    to~\eqref{eqn:PrePost_weighted_Precision} for a
    binary design $\design\in\{0, 1\}^\Nsens$.
    The robust OED approach proposed in this work (see~\Cref{sec:robust_oed}) follows the stochastic learning 
    approach in~\cite{attia2022stochastic} that utilizes only binary design values and avoids 
    relaxing the design, and thus the two forms~\eqref{eqn:PrePost_weighted_Precision} 
    and~\eqref{eqn:PrePost_weighted_Precision} can be equivalently utilized here. 
    
    If the problem is linear, the posterior is Gaussian~\eqref{eqn:Posterior_Params}
    with posterior covariance $\Cparampostmat$ independent from the
    prior mean and the actual data instances. Thus, an optimal design can be
    defined as the one that optimizes a scalar summary of the posterior
    uncertainty before actual deployment of the observational sensors.
    Generally speaking, in linear Gaussian settings, most OED utility
    functions, such as those utilizing the Fisher information matrix (FIM) or the
    posterior covariance matrix, are independent from the prior mean and the
    actual data instances and depend only on the simulation model and the 
    uncertainties prescribed by the prior and the observational noise. 
    Popular utility functions include 
    $\utilityfunc(\design):=\Trace{\FIM(\design)}$ for \emph{A-optimal} design and 
    $\utilityfunc(\design):=\logdet{\FIM(\design)}$ for \emph{D-optimal} design.
    In the linear Gaussian case, the information matrix $\FIM$ is equal to
    the inverse of the posterior covariance matrix~\eqref{eqn:Posterior_Params}, 
    that is,
    $
    \FIM(\design) \equiv
        \Cparampostmat\inv(\design)
          = \F\adj \wdesignmat(\design) \F
            + \Cparampriormat\inv \,,
    $
    where $\F$ and $\F\adj$ are the linear forward (parameter-to-observable) operator 
    and the associated adjoint, respectively.

    When the model $\Fcont$ is nonlinear, however, the FIM requires evaluating the
    tangent-linear model at the true parameter, that is,
    $\F=\partial\Fcont|_{\iparam=\ipartrue}$. Thus, to obtain an optimal design,
    one can iterate over finding the maximum a posteriori (MAP) estimate of $\iparam$
    and solving an OED problem using a Gaussian approximation around that estimate.
    This approach coincides with iteratively solving the OED problem by using a
    Laplacian approximation, where the posterior is approximated by using a
    Gaussian centered around the MAP estimate~\cite{attia2018goal}.
    This approach, however, requires solving the inverse problem for a specific choice
    of the prior mean (assuming Gaussian prior) and a given data generation
    process.
    Note that when the FIM is used, the optimal design is defined as a maximizer of
    the information content, while in the case of posterior uncertainties, the
    optimal design is defined by solving a minimization problem. 
    As mentioned in~\Cref{subsec:background_oed}, in both cases 
    one has to define the sign of the regularization accordingly.
    In this work, for clarity we restrict the discussion to the former case; that is, the
    optimal design is defined as in~\eqref{eqn:binary_OED_optimization}.
    Another approach for handling nonlinearity and/or non-Gaussianity is to use the 
    KL divergence between the prior and the posterior as a utility function~\cite{HuanMarzouk13}. 
    
    As mentioned in~\Cref{sec:introduction}, in this work we focus on
    A-optimal designs for linear/linearized Bayesian OED problems; 
    that is, we define the utility function as the trace of the $\FIM$:
    \begin{equation}\label{eqn:A_opt_utilityfunction}
      \utilityfunc(\design) := \Trace{\FIM(\design)}
      := 
        \Trace{
          \F\adj \pseudoinverse{ \designmat(\design) \odot \Cobsnoise } \F + \Cparampriormat\inv 
      }  \,,
    \end{equation}
    where the weighted precision matrix $\wdesignmat(\design)$ is defined by~\eqref{eqn:PrePost_weighted_Precision}.

    \commentout{Once the discussion in~\Cref{sec:robust_oed} is finalized; maybe say explicitly here 
      what challenges are associated with 
      other criteria such as D-optimality and how to possible handle them. For example,
      randomized estimates, form of the gradients, etc.
      }

  \subsection{Need for robustness: inverse problem misspecification}
  \label{subsec:IP_misspecification}
    In Bayesian inversion and Bayesian OED, the optimal design $\design\opt$ relies heavily on
    the degree of accuracy of the prior and the observation error model. 
    Specifically, a utility function that utilizes the FIM or the posterior covariance is
    greatly influenced by the accuracy of the prior covariance $\Cparamprior$
    and the observation error covariances $\Cobsnoise$.
    Prior selection remains a challenge in a vast range of applications.
    Additionally, the observation error covariance $\Cobsnoise$ is specified
    based on both instrumental errors and representativeness errors,
    that is, the inaccuracy of the observation operator that maps the model state onto
    the observation space. 
    Thus, in real-world applications, both the prior and the observation error uncertainties 
    can be easily misspecified or be known only to some degree a priori.
    Such misspecification or uncertainty can alter the utility value
    corresponding to each candidate design.
    One can specify a nominal value of the uncertain parameter and use it to solve the OED problem.
    The resulting design, however, is optimal only  for that specific choice and can lead to poor 
    results for cases other than the nominal value.
    Thus, such misspecification must be properly considered while solving an OED problem.

    In~\Cref{sec:robust_oed} we formulate the OED problem such that the resulting design
    is robust with respect to misspecification of the prior covariances
    $\Cparampriormat$ or the observation error covariances $\Cobsnoise$.
    We also provide algorithmic procedures to solve the formulated robust OED
    problems.

\section{Robust Bayesian OED}
\label{sec:robust_oed}
  In order to properly consider misspecification of inverse problem elements or parameters, 
  an admissible set must be properly characterized for the
  misspecified/uncertain parameter.
  In this work we focus on the prior covariances $\Cparamprior$ and the
  observation error covariances to be the uncertain parameters, and we study each
  case independently.
  To this end, we assume the prior covariance is parameterized as
  $\Cparamprior:=\Cparamprior(\uncertainparam^{\rm pr})$, where 
  $\uncertainparam^{\rm pr}\in\uncertainparamset^{\rm pr}$. 
  Similarly, we assume the observation error covariance is parameterized as
  $\Cobsnoise:=\Cobsnoise(\uncertainparam^{\rm noise})$, where 
  $\uncertainparam^{\rm noise}\in\uncertainparamset^{\rm noise}$. 
  Thus, the utility function takes the form
  $\utilityfunc(\design,\, \uncertainparam^{\rm pr},\,\uncertainparam^{\rm noise})$; 
  and because $\obj$ is a regularized version of the utility function $\utilityfunc$, 
  the objective function also takes the form $\obj(\design,\, \uncertainparam^{\rm pr},\,\uncertainparam^{\rm noise})$.
  Note that the parameterized form must assure that the resulting covariance
  operator/matrix is symmetric positive definite.
  
  To keep the discussion more general, we denote the uncertain parameter by
  $\uncertainparam\in\uncertainparamset$, where $\uncertainparamset$ is the admissible set of $\uncertainparam$. 
  Here $\uncertainparam$ could be the parameter of the prior
  covariance $\uncertainparam=\uncertainparam^{\rm pr}$ or the data
  uncertainty parameter $\uncertainparam=\uncertainparam^{\rm noise}$ or both
  $\left( {\uncertainparam^{\rm pr}}\tran,\, {\uncertainparam^{\rm noise} }\tran \right)\tran$.
  Thus the utility function and the regularized objective take the general 
  forms $\utilityfunc(\uncertainparam)$ and $\obj(\uncertainparam)$, respectively. 

  We start with a discussion that provides insight into the effect of robustifying the binary optimization 
  problem~\eqref{eqn:binary_OED_optimization}
  by optimizing against the worst-case scenario (max-min) in~\Cref{subsec:robust_binary_OED_insight}, followed by 
  the proposed approach starting from~\Cref{subsec:robust_oed_algorithm_stochastic}.

  \subsection{Insight on robustifying the binary optimization problem}
  \label{subsec:robust_binary_OED_insight}
    As mentioned in~\Cref{sec:introduction}, we follow a conservative 
    approach (Wald's maxmin model~\cite{wald1945statistical,wald1950statistical}) 
    to formulate and solve the robust OED problem.
    Specifically, we are interested in formulating and solving a robust counterpart
    of~\eqref{eqn:binary_OED_optimization} to find an optimal design against
    the worst-case (WC) scenario.
    Thus, we define an optimal design that is robust with respect to
    misspecification of $\uncertainparam$ as
    \begin{equation}\label{eqn:binary_robust_OED_optimization}
    \design\WCopt
        = \argmax_{\design\in\{0,1\}^\Nsens}\,  
            \min_{\uncertainparam\in\uncertainparamset} 
              \obj(\design,\,\uncertainparam)
          \,.
    \end{equation}

    This approach seeks the optimal design $\design$ that maximizes the objective $\obj$
    against the worst-case-scenario as defined by the uncertain parameter, that is, the lower envelop defined by $\uncertainparam$.
    \Cref{fig:Binary_1D_Diagram} provides a simple diagrammatic explanation of robust binary optimization 
    for a one-dimensional binary design $\design\in \{0, 1\}$.
    \begin{figure}[htbp!]
    \center
      \includegraphics[width=0.50\textwidth]{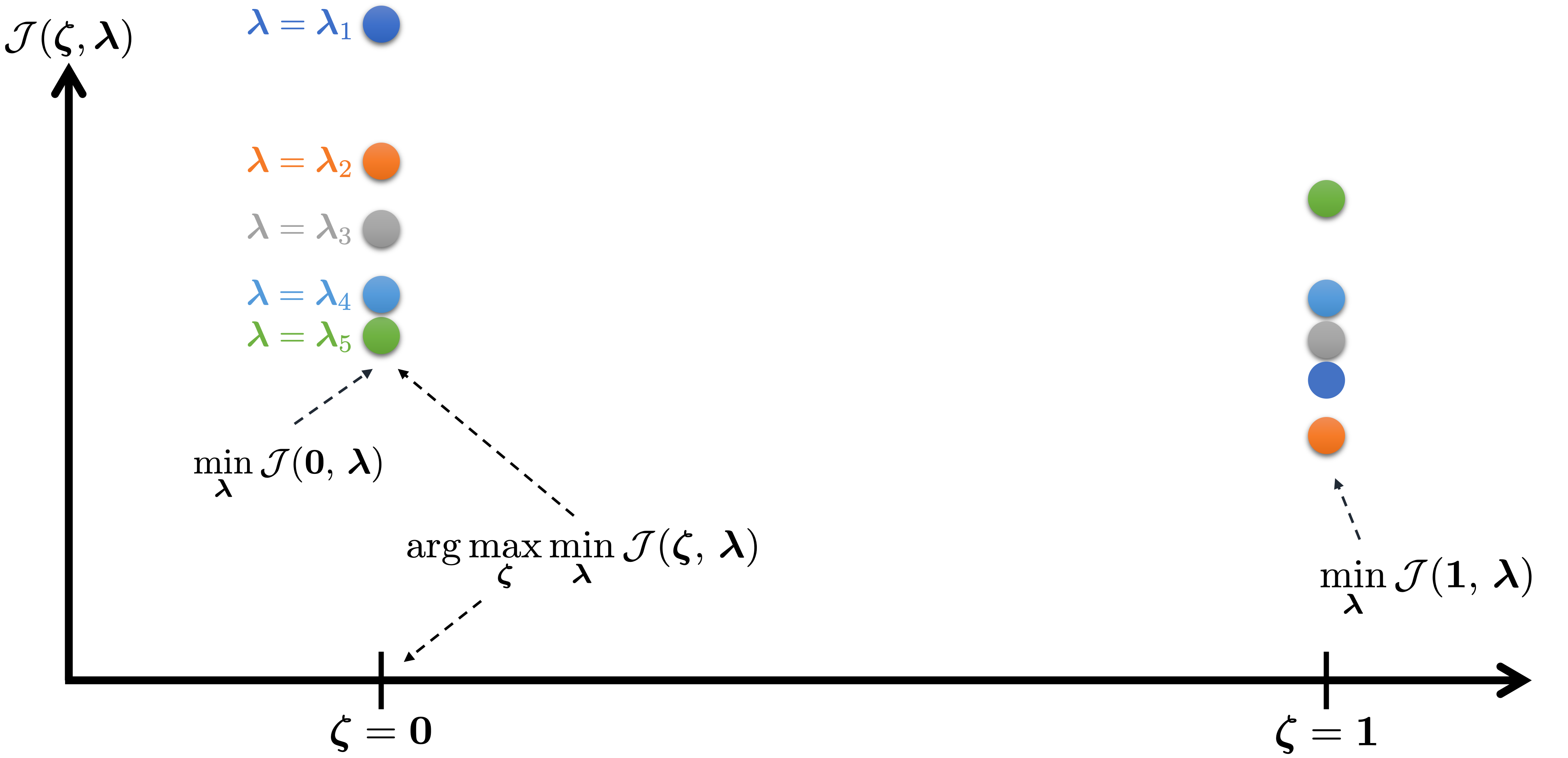}
      \caption{
        Robust binary optimization diagram.
        The design variable has two possible values, $\design\in\{0, 1\}$, and the uncertain 
        parameter is assumed to be finite for simplicity with $\uncertainparam\in\{\uncertainparam_i\,|\, i=1, \ldots, 5\}$. 
        The robust optimal design here is $\design\opt=0$ since it is the one that assumes the maximum of the lower bound 
        $\{\min\limits_{\uncertainparam}\obj(0, \uncertainparam),\, \min\limits_{\uncertainparam}\obj(1, \uncertainparam) \}$.
      }
      \label{fig:Binary_1D_Diagram}
    \end{figure}

    Note that choosing a nominal value for the uncertain parameter could lead to an incorrect solution. 
    For example, if one chooses to solve the binary optimization problem only for the nominal 
    value $\uncertainparam=\uncertainparam_5$, 
    the optimal solution, that is, the design that maximizes the objective, is $\design=1$, which is not the same as the 
    robust optimal design $\design\opt=0$.
    Note also that the lower bound created by the uncertain parameter
    $\{\min\limits_{\uncertainparam}\obj(0, \uncertainparam),\, \min\limits_{\uncertainparam}\obj(1, \uncertainparam) \}$
    is not formed by a single value of the uncertain parameter. Thus, restricting the solution of the OED optimization problem 
    to a specific nominal value is generally not expected to result in a robust optimal design. 
    
    We note, that for A-optimal robust designs, we can remove the dependence on the uncertain prior, $\uncertainparam^{\rm pr}$, because it does not affect the optimal design. This observation follows from the following equivalence:
    \begin{eqnarray}
    & & \displaystyle \max_{\design\in\{0,1\}^\Nsens}\,  
    \min_{\uncertainparam\in\uncertainparamset}
     \Trace{\FIM(\design)}
          := 
            \Trace{
              \F\adj \pseudoinverse{ \designmat(\design) \odot \Cobsnoise(\uncertainparam^{\rm noise} } \F + \Cparampriormat(\uncertainparam^{\rm pr})\inv 
          }  \nonumber \\
    \Leftrightarrow & & 
      \displaystyle \max_{\design\in\{0,1\}^\Nsens}\,  
    \min_{\uncertainparam\in\uncertainparamset}
          \left\{  \Trace{
              \F\adj \pseudoinverse{ \designmat(\design) \odot \Cobsnoise(\uncertainparam^{\rm noise} } \F} + 
              \Trace{\Cparampriormat(\uncertainparam^{\rm pr})\inv 
          } \right\}  \nonumber \\
          \Leftrightarrow & & 
      \displaystyle \max_{\design\in\{0,1\}^\Nsens}\,  
    \min_{\uncertainparam\in\uncertainparamset}
            \Trace{
              \F\adj \pseudoinverse{ \designmat(\design) \odot \Cobsnoise(\uncertainparam^{\rm noise} } \F} + 
    \min_{\uncertainparam\in\uncertainparamset}
              \Trace{\Cparampriormat(\uncertainparam^{\rm pr})\inv 
          } \; , \nonumber 
    \end{eqnarray}
    which follows from the linearity of the trace operator, and the fact that second trace is independent of $\design$. Therefore, we do not include uncertainty in $\Cparampriormat$ in our analysis. We do note, however, that for other robust designs, such as D-optimal robust design, uncertainty in the prior needs to be taken into account.

    \subsubsection{General algorithmic solution approach}
    \label{subsubsec:robust_OED_algorithm}
      %
      \Cref{alg:Polyak_MaxMin} is a general sampling-based approach due to~\cite{levitin1966constrained}
      for solving the max-min optimization problem~\eqref{eqn:binary_robust_OED_optimization} and other relaxations.
      \begin{algorithm}[htbp!]
        \setstretch{1.0}
        \caption{Algorithm for
          solving the max-min optimization problem~\eqref{eqn:binary_robust_OED_optimization}}
        \label{alg:Polyak_MaxMin}
        \begin{algorithmic}[1]
          \Require{
            Initial guess $\design^{(0)}$;
            and an initial sample  
              $\uncertainparamsample^{(k)}:=\{\uncertainparam^{(i)}\in\uncertainparamset|
              i=1, \ldots, k\!-\!1\} $ from the uncertain parameter admissible set.
          }

          \Ensure{$\design\WCopt$}

          \State{initialize $l = 0$}

          \While{Not Converged}

            \State\label{algstep:Polyak_MaxMin_design_opt}{$
                \design^{(l+1)} \!\leftarrow\! \argmax\limits_{\design} \,
                \min\limits_{\uncertainparam^{\prime}\in\uncertainparamsample^{(k+l-1)}} \obj(\design,\uncertainparam) 
              $ 
            } 

            \State\label{algstep:Polyak_MaxMin_param_opt}{$
            \uncertainparam^{(k+l)} \leftarrow \argmin\limits_{\uncertainparam \in \uncertainparamset}
                \obj(\design^{(l+1)},\uncertainparam)
              $ 
            } 

            \State{$
              \uncertainparamsample^{(k+l)} \leftarrow \uncertainparamsample^{(k+l-1)} \cup \{\uncertainparam^{(k+l)}\}
              $
            }
          
            \State{$l \leftarrow l+1$}

          \EndWhile

            \State{$\design\WCopt \leftarrow \design^{(l+1)}$}
  
            \State{\Return{$\design\WCopt$}}

        \end{algorithmic}

      \end{algorithm}

    \Cref{alg:Polyak_MaxMin} starts with an initial finite sample of the uncertain parameter 
    $\uncertainparamsample\subset\uncertainparamset$ and works by alternating the solution of two optimization problems.
    The first (Step~\ref{algstep:Polyak_MaxMin_design_opt}) is a maximization problem to update the target 
    variable (here the design) $\design$.
    The second (Step~\ref{algstep:Polyak_MaxMin_param_opt}) is a minimization problem to update the 
    uncertain parameter, which is then used to expand the sample $\uncertainparamsample$ until convergence.
    The max-min conservative approach for robust OED 
    can be traced back to the early
    works~\cite{walter1987robust,pronzato1988robust,pronzato2004minimax}; however,
    it was noted (see, e.g.,~\cite{walter1987robust,asprey2002designing}) that this approach has not been
    widely used, despite its importance, because of the computational overhead posed by the max-min
    optimization.  
    Solving~\eqref{eqn:binary_robust_OED_optimization}, for example, following the steps of~\Cref{alg:Polyak_MaxMin}, 
    would require enumerating all possible binary designs
    and, based on the nature of $\uncertainparam$, would also require probing the space of the uncertain parameter $\uncertainparam$.

    The two optimization steps (Steps \ref{algstep:Polyak_MaxMin_design_opt} and  \ref{algstep:Polyak_MaxMin_param_opt}) can be 
    carried out by following a gradient-based approach. Doing so, however,  would require applying relaxation of the binary design, 
    thus allowing all values of $\design\in[0, 1]$. 
    Moreover, Step~\ref{algstep:Polyak_MaxMin_design_opt} involves the derivative of the objective $\func(\design, \uncertainparam)$, 
    with respect to the target variable $\design$ over a finite sample $\uncertainparamsample$ of the uncertain parameter, which will 
    require introducing the notion of a generalized gradient~\cite{Rockafellar70}. However, we show below that by optimizing over the corresponding Bernoulli parameters, we can avoid the complication of generalised gradients.
    %

    %
    In our work we present and discuss an efficient approach to overcome such
    hurdles and yield robust optimal designs for sensor placement for Bayesian
    inverse problems. 
    In~\Cref{subsubsec:naive_robust_OED_formulations} we share our insight on naive attempts 
    to solve~\eqref{eqn:binary_robust_OED_optimization},
    followed by our proposed approach in~\Cref{subsec:robust_oed_algorithm_stochastic}.

    \subsubsection{Naive relaxation and stochastic formulations}
    \label{subsubsec:naive_robust_OED_formulations}
      %
      As discussed in~\Cref{subsec:background_oed}, to
      solve~\eqref{eqn:binary_robust_OED_optimization}, we can replace the outer optimization 
      problem with either a relaxation~\eqref{eqn:relaxed_OED_optimization} or the stochastic
      formulation~\eqref{eqn:stochastic_OED_optimization}. 
      This corresponds to replacing the optimization
      problem~\eqref{eqn:binary_robust_OED_optimization} with one of the following
      formulations:  
      \begin{subequations}\label{eqn:naive_robust_forms}
      \begin{align}
        \design\WCopt
            &= \argmax_{\design\in[0,1]^\Nsens}\, 
                \min_{\uncertainparam\in\uncertainparamset} 
                  \obj(\design,\,\uncertainparam)
              \label{eqn:robust_relaxed_OED_optimization}
                \\
        \hyperparam\WCopt_{\uncertainparam}
            &= \argmax_{\hyperparam\in[0,1]^\Nsens} \,
              \min_{\uncertainparam\in\uncertainparamset}  
              \stochobj(\hyperparam,\,\uncertainparam)
              :=\Expect{\design\sim\CondProb{\design}{\hyperparam}}{
                \obj(\design,\,\uncertainparam)
              }
              \label{eqn:robust_stochastic_OED_optimization_I}
            \,.
      \end{align}
      \end{subequations}

      In~\eqref{eqn:robust_relaxed_OED_optimization} the outer binary optimization
      problem is replaced with a relaxation over the relaxed design space, while
      in~\eqref{eqn:robust_stochastic_OED_optimization_I} the outer problem is
      replaced with a stochastic formulation with the objective of finding an
      optimal observational policy (activation probabilities) that is robust 
      with respect to the misspecified parameter $\uncertainparam$. 
      As discussed in~\cite{attia2022stochastic}, the relaxed
      form~\eqref{eqn:robust_relaxed_OED_optimization} requires the penalty term to
      be differentiable, and the solution procedure requires specifying the derivative of
      the utility function $\utilityfunc$ with respect to the design $\design$,
      while~\eqref{eqn:robust_stochastic_OED_optimization_I} does not require
      differentiability of the penalty term and does not require the derivative of the
      utility function with respect to the design.

      While these attempts to formulate the robust binary OED optimization problem are rather intuitive,  
      the two formulations~\eqref{eqn:robust_relaxed_OED_optimization} and
      \eqref{eqn:robust_stochastic_OED_optimization_I}  are neither equivalent to each other 
      nor equivalent to the original robust binary optimization 
      problem~\eqref{eqn:binary_robust_OED_optimization}.
      This is explained by~\Cref{fig:Naive_Relaxations_1D_Diagram}, which extends the example shown in~\Cref{fig:Binary_1D_Diagram} 
      by applying the formulations in~\eqref{eqn:robust_relaxed_OED_optimization} (left) 
      and~\eqref{eqn:robust_stochastic_OED_optimization_I} (right), respectively.
      \begin{figure}[htbp!]
      \center
        \includegraphics[width=0.49\textwidth]{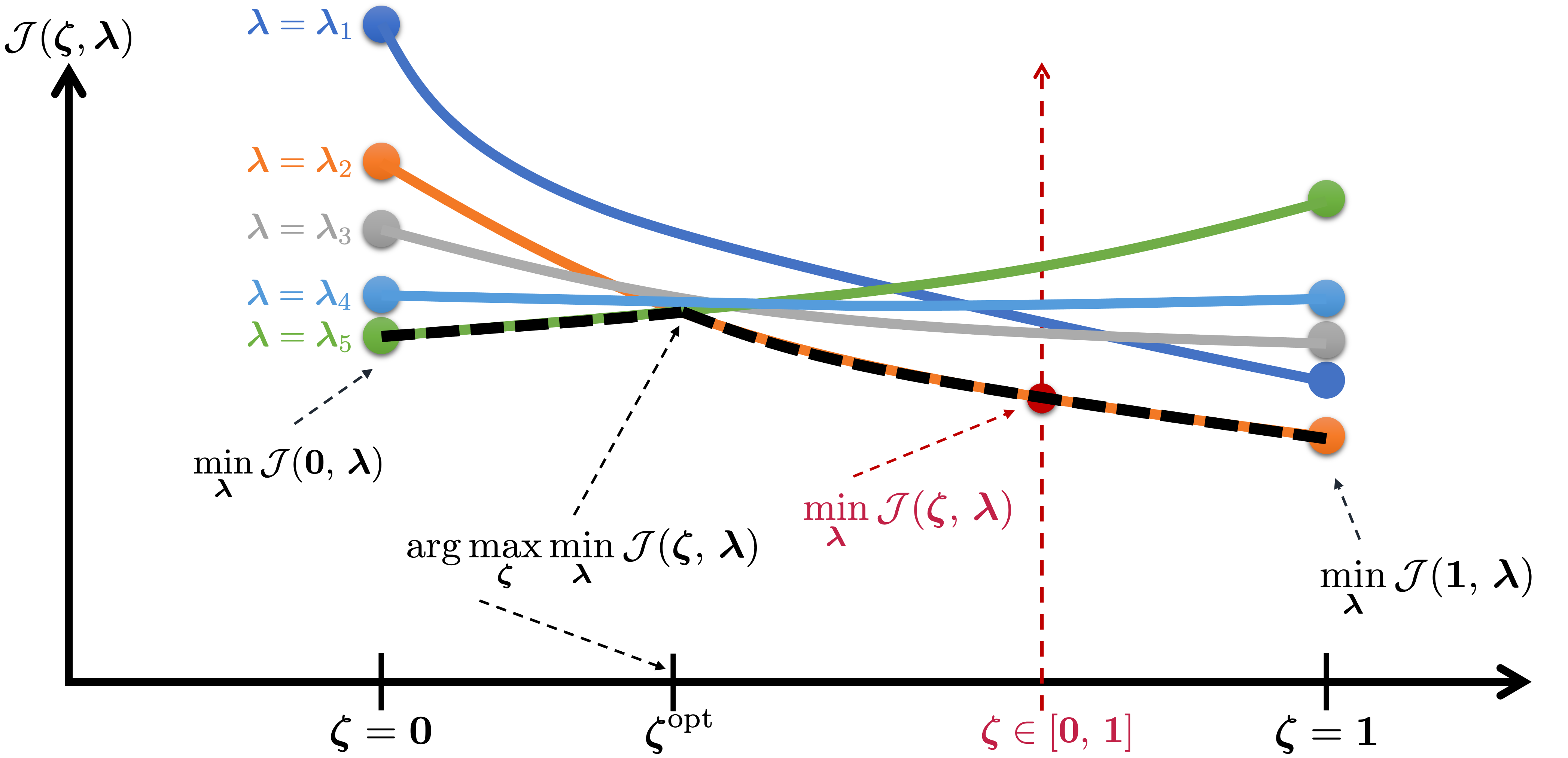}
          \hfill
        \includegraphics[width=0.49\textwidth]{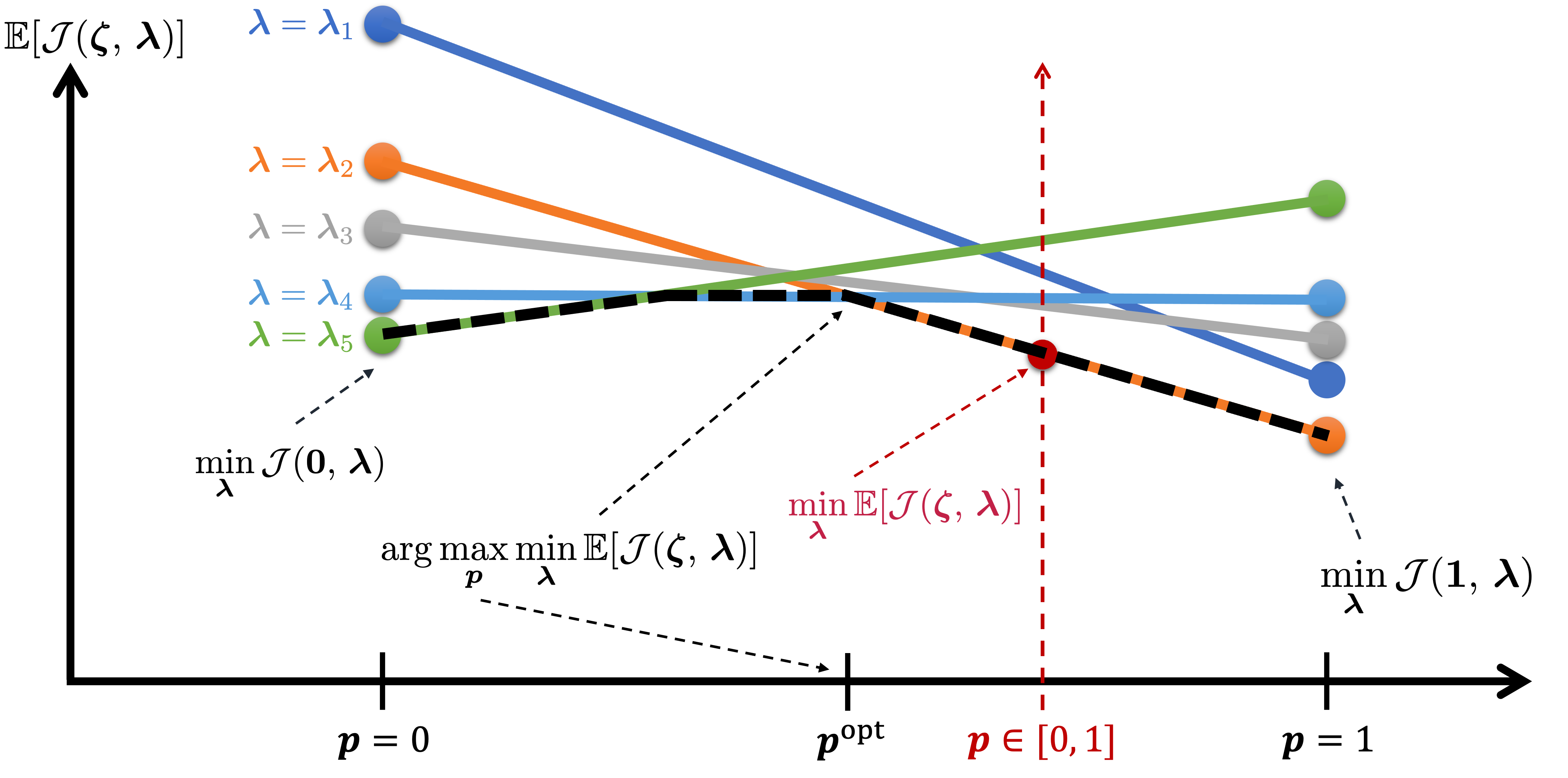}
          \caption{ Similar to~\Cref{fig:Binary_1D_Diagram}.
          Left: the design is relaxed to the domain $[0, 1]$ using~\eqref{eqn:robust_relaxed_OED_optimization}.
          Right: the stochastic formulation~\eqref{eqn:robust_stochastic_OED_optimization_I} is used.
          The optimal solutions of the two formulations~\eqref{eqn:naive_robust_forms} 
          are not equal, and both are different from the solution of 
          the original robust binary optimization problem~\eqref{eqn:binary_robust_OED_optimization}.
        }
        \label{fig:Naive_Relaxations_1D_Diagram}
      \end{figure}

      Note that although the lower-bound envelope constructed by relaxation is continuous everywhere,  
      differentiability with respect to the design $\design$ (or the Bernoulli parameter $\hyperparam$ in the stochastic formulation) 
       is not preserved in this example and thus is generally not guaranteed for the formulations~\eqref{eqn:naive_robust_forms}.
      This means that gradient-based optimization routines cannot be utilized to solve the outer optimization (max) problem.
      %
      Moreover, the optimal solution of both formulations~\eqref{eqn:naive_robust_forms} in this case are nonbinary, 
      which mandates applying rounding methodologies to obtain a binary estimate of the solution of the original 
      robust binary OED optimization problem~\eqref{eqn:binary_robust_OED_optimization}.
      In~\Cref{subsec:robust_oed_algorithm_stochastic} we present a formulation that overcomes such hurdles.

  \subsection{Robust binary optimization: a stochastic learning approach}
    \label{subsec:robust_oed_algorithm_stochastic}
    We formulate the stochastic robust binary optimization problem as follows:
    \begin{equation}\label{eqn:robust_stochastic_binary_optimization}
      \hyperparam\WCopt_{\uncertainparam}
        = \argmax_{\hyperparam\in[0,1]^{\Nsens}} \,
            \Expect{\design\sim\CondProb{\design}{\hyperparam}}{
              \min_{\uncertainparam\in\uncertainparamset}  
                \obj(\design, \uncertainparam)
        } \,.
    \end{equation}

    The form~\eqref{eqn:robust_stochastic_binary_optimization} constructs an expectation surface (smooth) that connects 
    the minimum values of the objective (over $\uncertainparam$) at all possible binary design variables $\design\in\{0, 1\}^{\Nsens}$.
    This fact is explained by using the example utilized in~\Cref{subsec:robust_binary_OED_insight} and \Cref{subsubsec:naive_robust_OED_formulations} 
    and is shown schematically in~\Cref{fig:Stochastic_Robust_1D_Diagram}.
    \begin{figure}[htbp!]
    \center
      \includegraphics[width=0.50\textwidth]{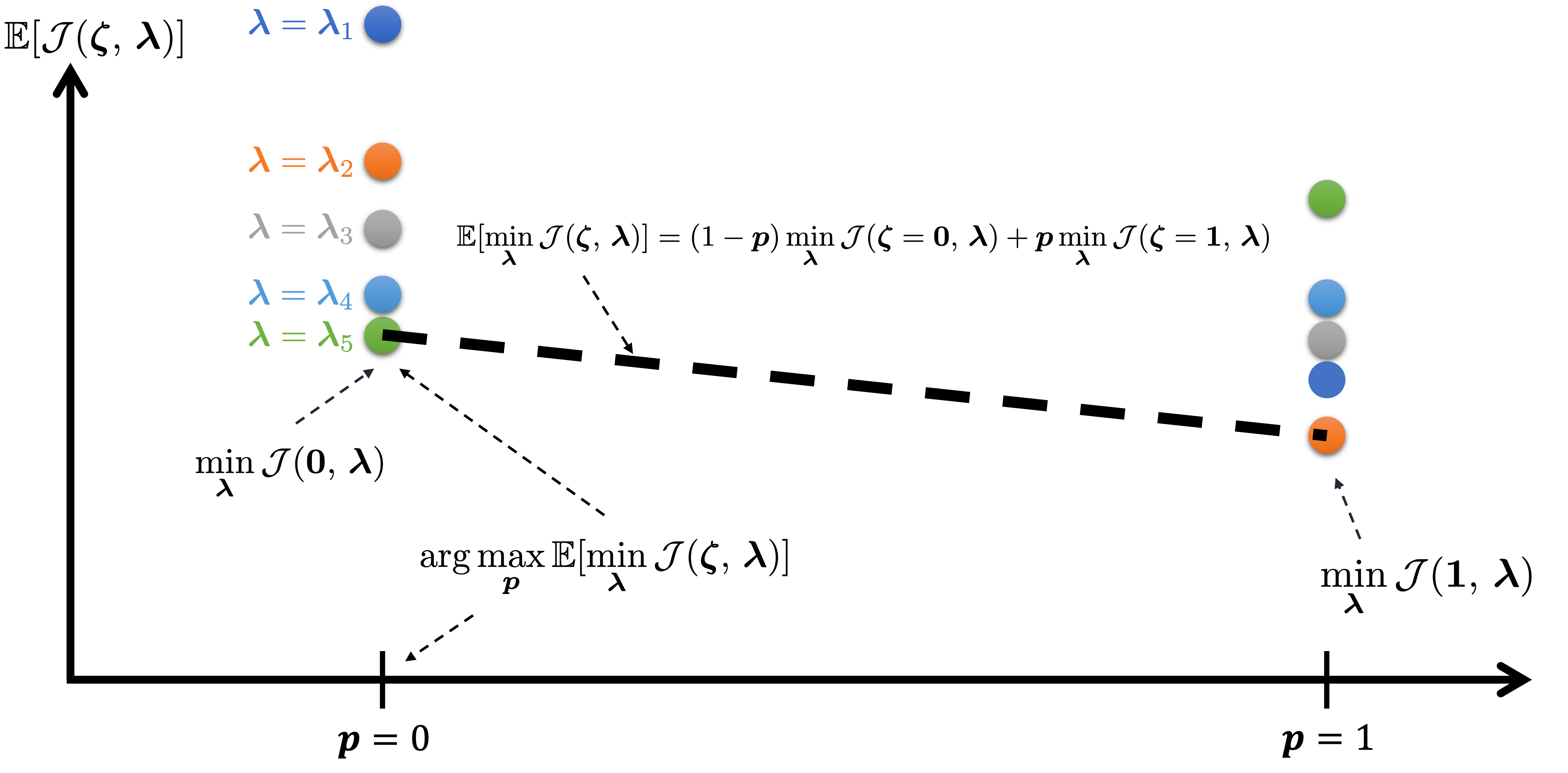}
      \caption{ Similar to~\Cref{fig:Naive_Relaxations_1D_Diagram}.
          Here, the stochastic formulation~\eqref{eqn:robust_stochastic_binary_optimization} is used.
      }
      \label{fig:Stochastic_Robust_1D_Diagram}
    \end{figure}

    This formulation can be applied to general robust binary optimization problems.
    Here, we discuss the formulation and the algorithmic approach in the context of binary OED problems for sensor placement as stated in the introduction.
    Specifically, we formulate the stochastic optimization problem for 
    robust A-OED~\eqref{eqn:robust_stochastic_A_optimal_OED_optimization} as follows:
    \begin{equation}\label{eqn:robust_stochastic_A_optimal_OED_optimization}
      \begin{aligned}
      \hyperparam\WCopt_{\uncertainparam}
        &= \argmax_{\hyperparam\in[0,1]^\Nsens} \,
            \Expect{\design\sim\CondProb{\design}{\hyperparam}}{
              \min_{\uncertainparam\in\uncertainparamset}  
                \obj(\design, \uncertainparam)
          }  \,, \\
      \obj(\design, \uncertainparam) 
        & = \Trace{\FIM(\design,\,\uncertainparam) }
                    - \regpenalty \, \penaltyfunction{\design}
            \,, \\
      \FIM(\design,\uncertainparam) 
          &= \Fadj \pseudoinverse{ \diag{\design} \Cobsnoise (\uncertainparam) \diag{\design} } \F 
              + \Cparampriormat\inv  \,, \\
      \end{aligned}
    \end{equation}
    where all design variables are binary, that is, $\design_i\in\{0, 1\};,\, i=1, \ldots, \Nsens$.
    The numerical solution of the optimization problem~\eqref{eqn:robust_stochastic_A_optimal_OED_optimization} 
    does not require the objective $\obj$ to be differentiable with respect to the design $\design$ in order to 
    solve the outer optimization problem. 
    By definition, the objective is continuous and is differentiable with respect to the Bernoulli parameter $\hyperparam$.
    In order to solve the inner optimization problem, however, the objective $\obj$ must be differentiable with respect 
    to the uncertain parameter $\uncertainparam$.
    Note that the penalty function $\penaltyfunc$ is independent from the uncertain parameter and is thus 
    not required to be differentiable.
    
    To solve~\eqref{eqn:robust_stochastic_A_optimal_OED_optimization}, we 
    extend~\Cref{alg:Polyak_MaxMin} by 
    defining the objective function of the outer optimization problem as the penalized robust A-optimal OED 
    stochastic objective 
    $$
    \stochobj(\hyperparam, \uncertainparam):=\Expect{\design\sim\CondProb{\design}{\hyperparam}}{
        \min_{\uncertainparam} \obj(\design,\uncertainparam)
      }
    $$
    stated in~\eqref{eqn:robust_stochastic_A_optimal_OED_optimization}.
    Thus, we require developing the derivative of the stochastic objective 
    $ 
    \stochobj(\hyperparam, \uncertainparam)
    $ 
    with respect to 
    the Bernoulli distribution hyperparameter $\hyperparam$.
    %
    The derivative of the expectation with respect to the hyperparameter $\hyperparam$
    of the Bernoulli distribution is approximated by using the Kernel trick as
    explained in~\cite{attia2022stochastic}.
    Specifically, from~\cite[Equation 3.9]{attia2022stochastic}, it follows
    that
    \begin{equation}\label{eqn:stochastic_gradient_hyperparam}
      \begin{split}
        \nabla_{\hyperparam}
        \stochobj(\hyperparam,\, \uncertainparam)
        = \nabla_{\hyperparam}
        \Expect{\design\sim\CondProb{\design}{\hyperparam}} {
              \min_{\uncertainparam} \obj(\design,\uncertainparam)
        }
        &= \Expect{\design\sim\CondProb{\design}{\hyperparam}} {
          \min_{\uncertainparam} \obj(\design,\uncertainparam) 
          \nabla_{\hyperparam} \log \CondProb{\design}{\hyperparam}
        }
        \\
        &\approx
        \frac{1}{\Nens} \sum_{k=1}^{\Nens}
              \min_{\uncertainparam} \obj(\design[k],\uncertainparam)
          \nabla_{\hyperparam} \log \CondProb{\design[k]}{\hyperparam}
          \,,
      \end{split}
    \end{equation}
    given $\Nens>0$ samples 
    $
      \{\design[k] \sim \CondProb{\design}{\hyperparam}\, |\,\, k=1, \ldots, \Nens\}
    $
    drawn from a multivariate Bernoulli distribution with
    parameter $\hyperparam$, and the gradient of the log-probability 
    $ 
    \nabla_{\hyperparam} \log \CondProb{\design[k]}{\hyperparam}
    $ of the Bernoulli distribution  is
    \begin{equation}\label{eqn:Bernoulli_grad_log_prob}
      \nabla_{\hyperparam} \log \CondProb{\design}{\hyperparam}
        =\sum_{i=1}^{\Nsens}{\left( \frac{\design_i}{\hyperparam_i} +
          \frac{\design_i-1}{1-\hyperparam_i} \right)\vec{e}_i}  \,; \quad \hyperparam_i\in(0, 1) \,,
    \end{equation}
    where $\vec{e}_i$ the $i$th unit vector in $\Rnum^{\Nsens}$ and $\hyperparam_i$ is the $i$th entry of $\hyperparam$.  
    
    As mentioned  in~\ref{subsubsec:robust_OED_algorithm}, 
    in order to solve the outer (maximization) problem, \Cref{alg:Polyak_MaxMin} requires developing 
    the  gradient over a finite sample 
    $\uncertainparamsample\subset\uncertainparamset$ of the uncertain parameter.
    The gradient in this case is given by 
    \begin{subequations}\label{eqn:eqn:discrete_stochastic_gradien_full}
    \begin{align}
        &\nabla_{\hyperparam} 
          \Expect{\design\sim\CondProb{\design}{\hyperparam}}{
              \min\limits_{\uncertainparam\in\uncertainparamsample} \obj(\design,\uncertainparam)
          }
      \stackrel{\eqref{eqn:stochastic_gradient_hyperparam}}{\approx }
        \frac{1}{\Nens} \sum_{k=1}^{\Nens}
              \obj(\design[k],\uncertainparam^{*}) 
      \nabla_{\hyperparam} \log \CondProb{\design}{\hyperparam}
          \,,  \label{eqn:discrete_stochastic_gradient}
          \\
      &\qquad \uncertainparam^{*}
        \in  \argmin\limits_{\uncertainparam\in\uncertainparamsample}
            \obj(\design,\uncertainparam)
          \,.
    \end{align}
    \end{subequations}

    Note that both~\eqref{eqn:stochastic_gradient_hyperparam} and~\eqref{eqn:eqn:discrete_stochastic_gradien_full}
    require finding $\min_{\uncertainparam}\obj(\design[k],\, \uncertainparam)$, which is 
    the minimum value of the objective $\obj$ (over a finite number of $\uncertainparam$) for a nominal binary design $\design[k]$.
    %
    Conversely, the numerical solution of the inner optimization problem in
    Step~\ref{algstep:Polyak_MaxMin_param_opt}
    of~\Cref{alg:Polyak_MaxMin} requires the gradient 
    $
     \nabla_{\uncertainparam} \obj(\design,\uncertainparam)
      = \nabla_{\uncertainparam} \utilityfunc(\design,\uncertainparam)$, which requires
    specifying the dependency of $\obj$ on $\uncertainparam$.
    %
    The gradient with respect to $\uncertainparam=\uncertainparam^{\rm noise}$
    is derived in
    \Cref{app:subsec:Robust_A-opt_utility_gradient:param}
    and is summarized by
    %
    \begin{align}
      %
      \nabla_{\uncertainparam^{\rm noise}} \obj(\design, \uncertainparam ) 
        &= -\! \sum_{i=1}^{\Nsens}{
          \Trace{ \! 
            \F\adj
              \wdesignmat
                 \diag{\design}  
                 \del{  \Cobsnoise(\uncertainparam^{\rm noise} ) } {{\uncertainparam^{\rm noise}_i}}  
                   \diag{\design} 
            \wdesignmat
              \F
            \!} 
            \vec{e}_i
          }
        \,,\label{eqn:utility_A_opt_gradient_obs_noise_relaxed}
    \end{align}
    where $\wdesignmat= \wdesignmat(\design;\uncertainparam^{\rm noise})$ is given 
    by~\eqref{eqn:PrePost_weighted_Precision}.
    Note that the derivatives of the parameterized 
    observation error 
    covariance
    matrix in~\eqref{eqn:utility_A_opt_gradient_obs_noise_relaxed}
    is
    application dependent
    and thus will be discussed further in the context of numerical experiments; see~\Cref{sec:numerical_results}.

    %
    \paragraph{Performance enhancement via variance reduction}
      As explained in~\cite{attia2022stochastic}, the performance of an algorithm that utilizes 
      a stochastic gradient approximation can be significantly enhanced by reducing the variability
      of the stochastic gradient estimator. This can be achieved, for example, by using 
      antithetic variates or importance sampling and introducing a baseline $\baseline$ constant 
      to the objective that replaces
      $\obj(\design, \uncertainparam)$ with $\obj(\design, \uncertainparam)-\baseline$.
      The optimal value $\baseline\opt$ of the baseline is then chosen to minimize
      the total variance 
      of the stochastic estimate of the gradient.
      In this case
      (see the discussion on the optimal baseline in~\cite{attia2022stochastic}),
      the  gradient~\eqref{eqn:discrete_stochastic_gradient} is replaced with 
      \begin{subequations}\label{eqn:stoch_generalized_gradient_baseline}
      \begin{align}
        &\nabla_{\hyperparam}
          \Expect{\design\sim\CondProb{\design}{\hyperparam}}{
          \min\limits_{\uncertainparam\in\uncertainparamsample}
              \obj(\design,\uncertainparam) - \baseline\opt
          }
        \stackrel{\eqref{eqn:stochastic_gradient_hyperparam}}{\approx }
        \frac{1}{\Nens} \sum_{k=1}^{\Nens}
          \left(
            \obj(\design[k], \uncertainparam^{*}[k]) - \baseline\opt 
          \right)
        \nabla_{\hyperparam} \log \CondProb{\design}{\hyperparam}
          \,,  \label{eqn:discrete_stochastic_gradient_baseline}
          \\
        \baseline\opt
        \approx
        \widehat{\baseline}\opt
          &:=
          \frac{\sum\limits_{e=1}^{{\rm N_b}} 
            \left(
              \sum\limits_{j=1}^{\Nens} \obj(\design[e,j],\uncertainparam^{*}[e,j]) \nabla_{\hyperparam}
              \log{ \CondProb{\design[e, j]}{\hyperparam } } 
            \right)
            \tran\! 
            \left(
              \sum\limits_{j=1}^{\Nens} \nabla_{\hyperparam}
              \log{ \CondProb{\design[e, j]}{\hyperparam } } 
            \right)
            }
          { \Nens\, {\rm N_b}\, \sum\limits_{i=1}^{\Nsens}\frac{1}{\hyperparam_i-\hyperparam_i^2}}
        \,,\label{eqn:optimal_baseline_estimate} \\
        &\qquad \uncertainparam^{*}[k]
          = \argmin\limits_{\uncertainparam\in\uncertainparamsample}
            \obj(\design[k],\uncertainparam)
          = \argmin\limits_{\uncertainparam\in\uncertainparamsample}
          \Trace{\FIM(\design[k],\,\uncertainparam) }
            \label{eqn:discrete_stochastic_gradient_opt_lambda}
          \,,
      \end{align}
      \end{subequations}
      where we introduced the baseline to reduce the variability of the stochastic estimator
      as in~\eqref{eqn:discrete_stochastic_gradient_baseline}.
      In this case, 
      $\uncertainparam^{*}[k]$ is the uncertain parameter $\uncertainparam$ that attains the minimum value of $\obj$ over a finite sample $\uncertainparamsample$ for the $k$th binary design $\design[k]$ in the sample, 
      rather than minimizing a sample average approximation of the expected value 
      $\Expect{}{\obj}$.
      This reduces the stochasticity in the baseline itself and thus reduces the variability 
      of the stochastic gradient, which in turn leads to better performance of the optimization 
      algorithm.

    \paragraph{Complete algorithm statement}
      We conclude~\Cref{subsec:robust_oed_algorithm_stochastic} with a complete algorithmic description of the 
      robust stochastic steepest ascent approach in~\Cref{alg:Polyak_MaxMin_Stochastic}.

      %
      \begin{algorithm}[htbp!]
          \setstretch{1.0}
          \caption{Algorithm for
              solving the robust stochastic OED problem~\eqref{eqn:robust_stochastic_A_optimal_OED_optimization}}
          \label{alg:Polyak_MaxMin_Stochastic}
          \begin{algorithmic}[1]
            \Require{
              Initial policy parameter $\hyperparam^{(0)}\in(0, 1)^{\Nsens}$;
                  step sizes ($\gamma_1,\, \gamma_2$);
                  sample sizes $\Nens, {\rm N_b}, m$;
              and an initial sample  
                $\uncertainparamsample^{(k)}:=\{\uncertainparam^{(i)}\in\uncertainparamset|
                i=1, \ldots, k\} $ from the uncertain parameter admissible set.
            }

            \Ensure{$\design\WCopt$}

            \State{Initialize $l = 0$}

            \While{Not Converged}

              \State{Update $ l \leftarrow l+1 $ }

              \State\label{algstep:Polyak_MaxMin_Stochastic_hyperparam_opt}{
                $
                  \hyperparam^{(l+1)} \!\leftarrow\! \argmax\limits_{\hyperparam}
                      \Expect{\design\sim\CondProb{\design}{\hyperparam^{(l)}}}{
                        \min\limits_{\uncertainparam^{\prime}\in\uncertainparamsample^{(k+l-1)}}
                        \obj(\design,\uncertainparam)
                    }
                $ \Comment{Call \textsc{StocParamOpt} } 
              }

              \State\label{algstep:Polyak_MaxMin_Stochastic_param_opt}{Update 
                  $\uncertainparam^{(k+l)} 
                    \leftarrow
                    \argmin\limits_{\uncertainparam \in \uncertainparamset}
                        \Expect{\design\sim\CondProb{\design}{\hyperparam^{(l+1)}}}{
                            \obj(\design,\uncertainparam)
                        }
                        $
                    \Comment{Call
                        $
                            \textsc{NoiseOpt}  
                        $
                    }
               }

              \State{
              $
                \uncertainparamsample^{(k+l)} \leftarrow \uncertainparamsample^{(k+l-1)} \cup \{\uncertainparam^{(k+l)}\}
              $
              }

            \EndWhile

            \State{$\hyperparam\WCopt_{\uncertainparam} \leftarrow \hyperparam^{(l+1)}$}

            \State{Sample $S=\{ \design[j] \sim \CondProb{\design}{\hyperparam\WCopt_{\uncertainparam}};\,
                j=1,\ldots,m \}$,
            }

            \State{Calculate $\obj(\design,\,\uncertainparam^{k+l})$ for each $\design$ in the sample $S$}

            \State\Return{$\design\WCopt_{\uncertainparam} $: the design $\design$ with largest 
              value of $\obj$ in the sample $S$.
            }

            \vspace{8pt}
            \Procedure{StocParamOpt}{$\gamma_1, \hyperparam^{(0)}, \uncertainparamsample, {\rm N_b}, \Nens$}

                \State{Initialize $n = 0$}
                
                \While{Not Converged}

                    \State{Sample
                        $S = \{\design[j] \sim \CondProb{\design}{\hyperparam^{(n)}};\,
                            j=1,\ldots,\Nens\!\times\!{\rm N_b}  \}$
                    }
                    
                    \State{Evaluate 
                        $
                            \nabla_{\hyperparam} \log \CondProb{\design}{\hyperparam} \,\,\forall\ \design \in S \,
                        $ using~\eqref{eqn:Bernoulli_grad_log_prob}
                    }
                    
                    \For{each design $\design[k]$ in $S$ }

                        \State\label{algstep:Polyak_MaxMin_Stochastic_inner_min}{Evaluate 
                        $
                            \{\obj(\design[k], \uncertainparam) \,\,
                        \forall \uncertainparam\in\uncertainparamsample\,;\quad \uncertainparam^{*}[k]$ 
                        corresponds to the smallest value. 
                        }  \Comment{\eqref{eqn:discrete_stochastic_gradient_opt_lambda}}
                    \EndFor

                    \State{Evaluate optimal baseline estimate $\widehat{\baseline}\opt$ 
                        using~\eqref{eqn:optimal_baseline_estimate}}
                   
                    \State{Evaluate 
                        $
                            \widehat{\vec{g}}^{(n)} 
                                = \frac{1}{|S|} \sum_{\design\in S}^{}
                                \left( \obj(\design,\uncertainparam^{*}) - \widehat{\baseline}\opt \right) 
                                \nabla_{\hyperparam} \log \CondProb{\design}{\hyperparam}
                        $ 
                    }\Comment{\eqref{eqn:discrete_stochastic_gradient_baseline}}

                    \State\label{algstep:Polyak_MaxMin_Stochastic_hyperparm_update}{
                        $
                          \hyperparam^{(n+1)} \leftarrow \Trunc{0, 1}{\hyperparam^{(n)} 
                            + \gamma_1 \widehat{\vec{g}}^{(n)} } 
                        $
                    }
                    
                    \State{Update $ n \leftarrow n+1 $ }

                \EndWhile

              \State{$\hyperparam\opt \leftarrow \hyperparam^{(n)}$}

              \State{\Return{$\hyperparam\opt$}}

            \EndProcedure

              \vspace{5pt}
            
                %

                    






                %
                \vspace{5pt}
                \Function{NoiseOpt}{$\gamma_2, \hyperparam, \uncertainparamsample, \Nens$,}

                    \State\label{algstep:NoiseOpt_Begin}{Initialize $n = 0$}
                    \State{Initialize 
                        $\uncertainparam^{(0)} $ = mean($\uncertainparamsample$)}  
                    \State{Sample
                        $S = \{\design[j] \sim \CondProb{\design}{\hyperparam};\,
                            j=1,\ldots,\Nens  \}$
                    }
                    
                    \While{Not Converged}
        
                        \State{Update $ n \leftarrow n+1 $ }

                        \For{$j \gets 1$ to $\Nens$}
                            \State\label{algstep:Polyak_MaxMin_Stochastic_obs_noise_gradient}{
                                Calculate $\vec{g}^{(n)}[j] 
                                    = \nabla_{\uncertainparam} \obj(\design[j], \uncertainparam )$ 
                            }\Comment{Use \eqref{eqn:utility_A_opt_gradient_obs_noise_relaxed}}
                            
                        \EndFor

                        \State $\vec{g}^{(n)} \gets \frac{1}{\Nens}\sum_{j=1}^{\Nens} \vec{g}^{n}[j] $
                        
                        \State{Update $\uncertainparam^{(n+1)} \leftarrow \uncertainparam^{(n)} - \gamma_2 \vec{g}^{(n)} $
                        }

                    \EndWhile

                  \State\label{algstep:NoiseOpt_End}{\Return{$\uncertainparam\opt$}}

                \EndFunction
        
          \end{algorithmic}

      \end{algorithm}
      %

    At each iteration of~\Cref{alg:Polyak_MaxMin_Stochastic}, Step~\ref{algstep:Polyak_MaxMin_Stochastic_hyperparam_opt} 
    seeks an optimal policy $\hyperparam$ given the current finite sample of the uncertain parameter $\uncertainparamsample$, 
    then Step~\ref{algstep:Polyak_MaxMin_Stochastic_param_opt} expands $\uncertainparamsample$ by finding an optimal value of the uncertain 
    parameter $\uncertainparam$ given the policy defined by the current value of $\hyperparam$.
    Because the inner optimization problem in~\eqref{eqn:robust_stochastic_binary_optimization} seeks to find $\argmin_{\uncertainparam} \obj(\design,\uncertainparam)$ for the current policy parameter $\hyperparam$, one has to either find the minimum over all possible designs
    (which requires brute force enumeration and is thus impractical,) or follow a stochastic approach whether the inner optimization problem is equivalently replaced with $\argmin_{\uncertainparam} \Expect{\design\sim\CondProb{\design}{\hyperparam}}{\obj(\design,\uncertainparam)}$ which is then solved following a stochastic gradient approach 
    as described by Steps~\ref{algstep:NoiseOpt_Begin}--\ref{algstep:NoiseOpt_End} of~\Cref{alg:Polyak_MaxMin_Stochastic}.

    Note that in Step~\ref{algstep:Polyak_MaxMin_Stochastic_hyperparam_opt}
    of~\Cref{alg:Polyak_MaxMin_Stochastic}, the optimal policy can be
    degenerate, that is, the probabilities $\hyperparam_i\in\{0, 1\}$. 
    This makes the inner optimization problem~\ref{algstep:Polyak_MaxMin_Stochastic_param_opt} search only over this degenerate policy
    (the worst-case scenario is looked upon only at this corner).
    Hence, we may not want to solve the outer problem exactly.  
    Rather, we may want to take a only few steps toward the optimal policy.
    Moreover, because the expectation surface is not necessarily convex, the optimizer can 
    get stuck in a local optimum. 
    In this case, when the optimal solution is unique, it is guaranteed to be binary.
    Thus one can add a random perturbation to the resulting policy and proceed until 
    convergence to a degenerate policy is achieved. 
    If the optimal solution is nonunique, however, the optimal solution can be a nonbinary probability, which 
    is then sampled seeking a solution in the optimal set.

\section{Numerical Experiments}
\label{sec:numerical_results}
  %
  In this section we discuss numerical results of the proposed robust OED approach 
  obtained by using a two-dimensional advection problem that simulates the evolution 
  of a contaminant concentration in a closed and bounded domain.
  All numerical experiments presented here are carried out by using the extensible software package
  \pyoed~\cite{attia2023pyoed}, with scripts to recreate experiments in this work made publicly available from~\cite{attia2023pyoedRepo}.

  We utilize an experiment that simulates the evolution of 
  the concentration of a contaminant in a closed domain following a similar setup to that
  in~\cite{attia2022stochastic}.
  The goal is to optimize sensor placement for accurate 
  identification of the contaminant source, that is, the distribution of the contaminant at
  the initial time of the simulation, or for optimal data acquisition. 
  Optimality here is defined following an A-optimality 
  approach where the observation noise variance is not known exactly.
  Specifically, we assume the observation error covariance matrix $\Cobsnoise=\Cobsnoise(\uncertainparam)$ 
  is parameterized with some parameter $\uncertainparam\in\uncertainparamset$, where 
  $\uncertainparamset$ is known.
  We start by describing the experimental setup of the simulation and the observation models
  (hence the parameter-to-observable map), as well as the prior and the observation noise model.
  Then we describe the robust OED optimization problem and discuss the results and the 
  performance of~\Cref{alg:Polyak_MaxMin_Stochastic}.

  \subsection{Model setup: advection-diffusion}
    The experiment uses an advection-diffusion (AD) model to simulate the spatiotemporal 
    evolution of the contaminant field $u=\xcont(\vec{x}, t)$ 
    over a simulation timespan $[0, T]$ in a closed domain $\domain$.
    Specifically, the contaminant field $u = \xcont(\mathbf{x}, t)$ 
    is governed by the AD equation 
      \begin{equation}\label{eqn:advection_diffusion}
        \begin{aligned}
          \xcont_t - \kappa \Delta \xcont + \vec{v} \cdot \nabla \xcont &= 0     
            \quad \text{in } \domain \times (0,T],   \\
          \xcont(x,\,0) &= \theta \quad \text{in } \domain,  \\
          \kappa \nabla \xcont \cdot \vec{n} &= 0  
            \quad \text{on } \partial \domain \times [0,T]\,,
        \end{aligned}
    \end{equation}
    where $\kappa\!>\!0$ is the diffusivity, $\vec{v}$ is the velocity field, and 
    $\domain\!=\![0, 1]^2$ is the spatial domain with two rectangular regions
    simulating two buildings inside the domain where the contaminant flow does not pass 
    through or enter. 
    The domain boundary $\partial \domain$ includes both the external boundary and the 
    walls of the two buildings.

    The finite-element discretization of the AD model~\eqref{eqn:advection_diffusion} 
    is sketched in~\Cref{fig:AD_Setup} (left).
    The ground truth of the contaminant source, that is, the initial distribution of
    the contaminant used in our experiments, is shown in~\Cref{fig:AD_Setup} (right).
    Following the same setup in~\cite{PetraStadler11,attia2022stochastic}, the velocity field 
    $\vec{v}$ is assumed to be known and is obtained by solving a steady-state Navier--Stokes equation
    with the side walls driving the flow; see~\Cref{fig:AD_Setup} (middle).
    \begin{figure}[htbp!]
      \centering
      \includegraphics[width=0.25\linewidth]{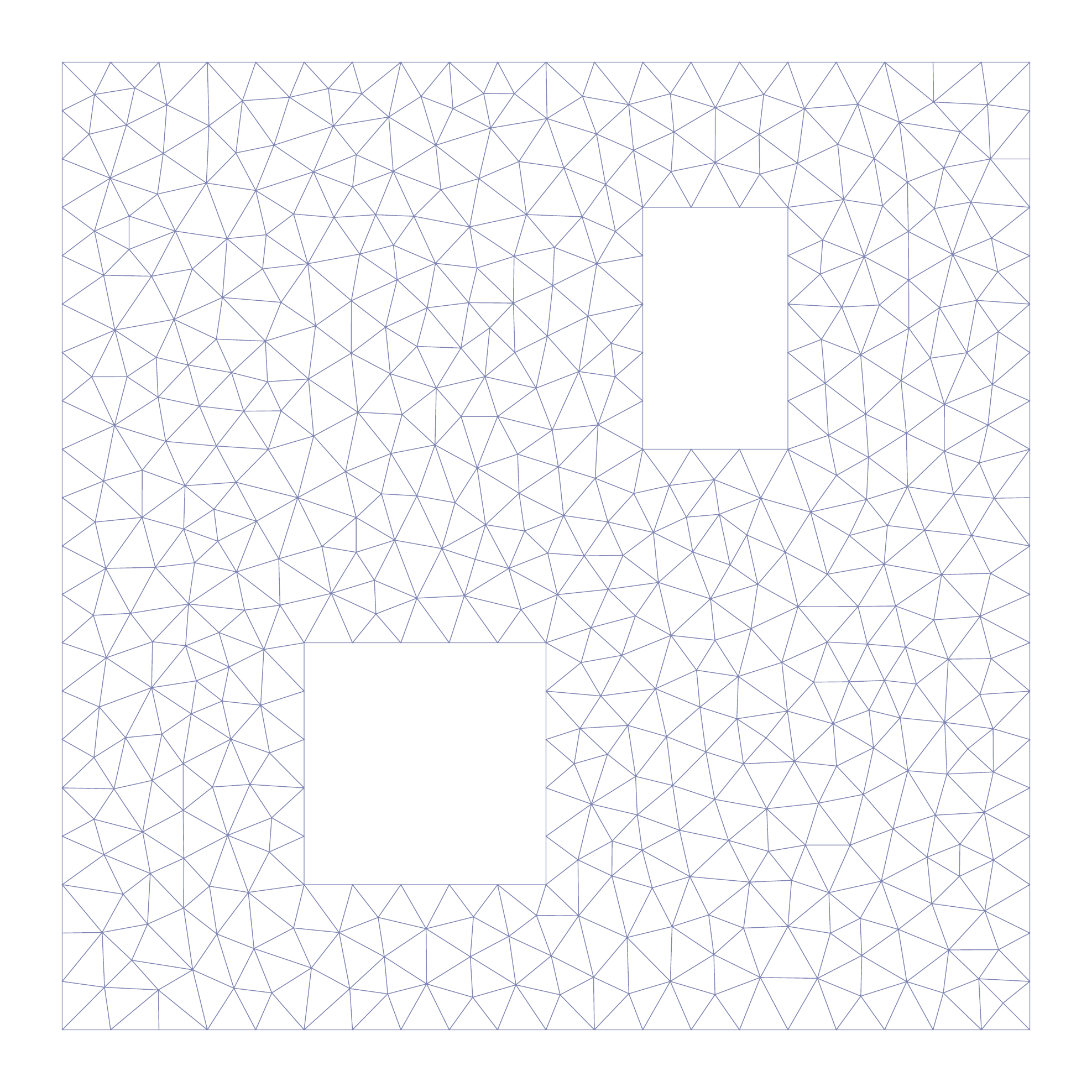}
      \quad
      \includegraphics[width=0.25\linewidth]{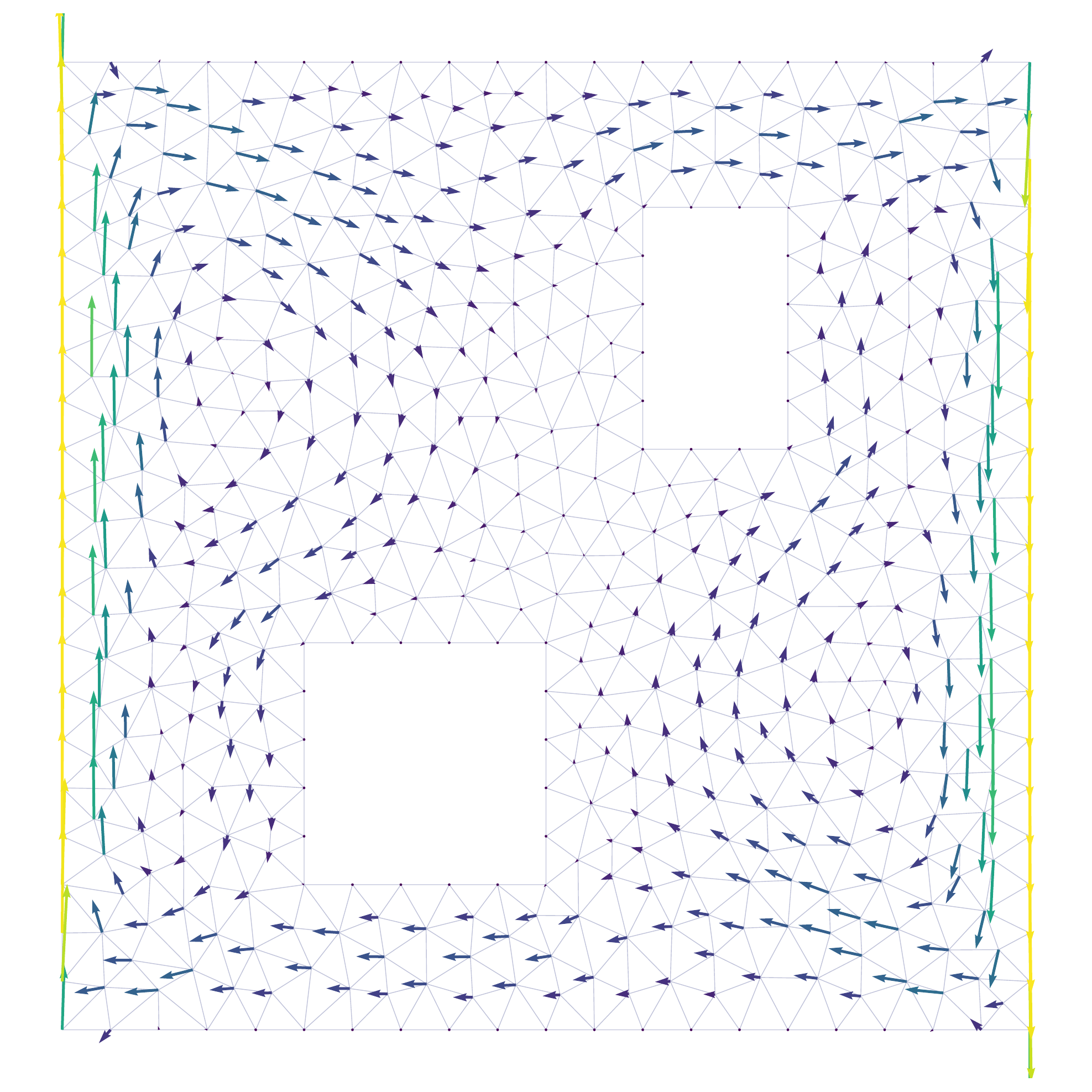}
      \quad
      \includegraphics[width=0.25\linewidth]{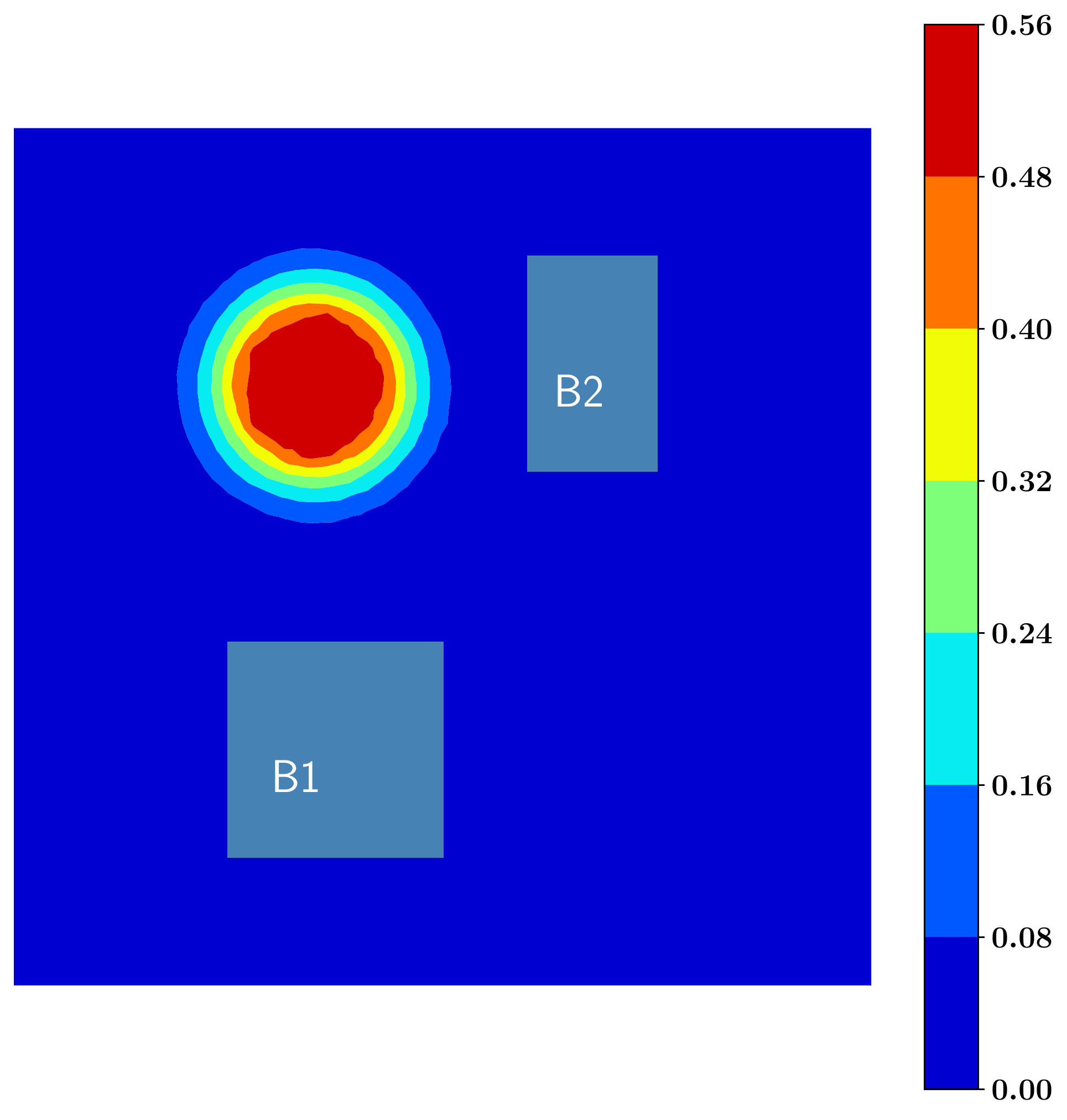}
        \caption{Advection-diffusion model~\eqref{eqn:advection_diffusion} domain and 
          discretization, the velocity field, and the ground truth 
          (the true initial condition).
          The two squares represent the two buildings (labeled B1 and B2, respectively) where the flow 
          is not allowed to enter.
          Left: The physical domain $\domain$, including the outer boundary 
          and the two buildings, and the model grid discretization.
          Middle: The constant velocity field $\vec{v}$.
          Right: The true initial condition (the contaminant concentration at the initial time).
        }
      \label{fig:AD_Setup} 
    \end{figure}
    %

  \subsection{Observational setup}
    We consider a set of fixed candidate locations; that is, the sensor locations do not change over time.
    First we consider the case where only two candidate sensors are allowed, 
    and then we consider an increasing number $\Nsens$ of candidate spatial observational gridpoints distributed 
    in the domain; see~\Cref{fig:AD_Observatin_Grids}.
    \begin{figure}[htbp!]
      \centering
      \includegraphics[width=0.25\linewidth]{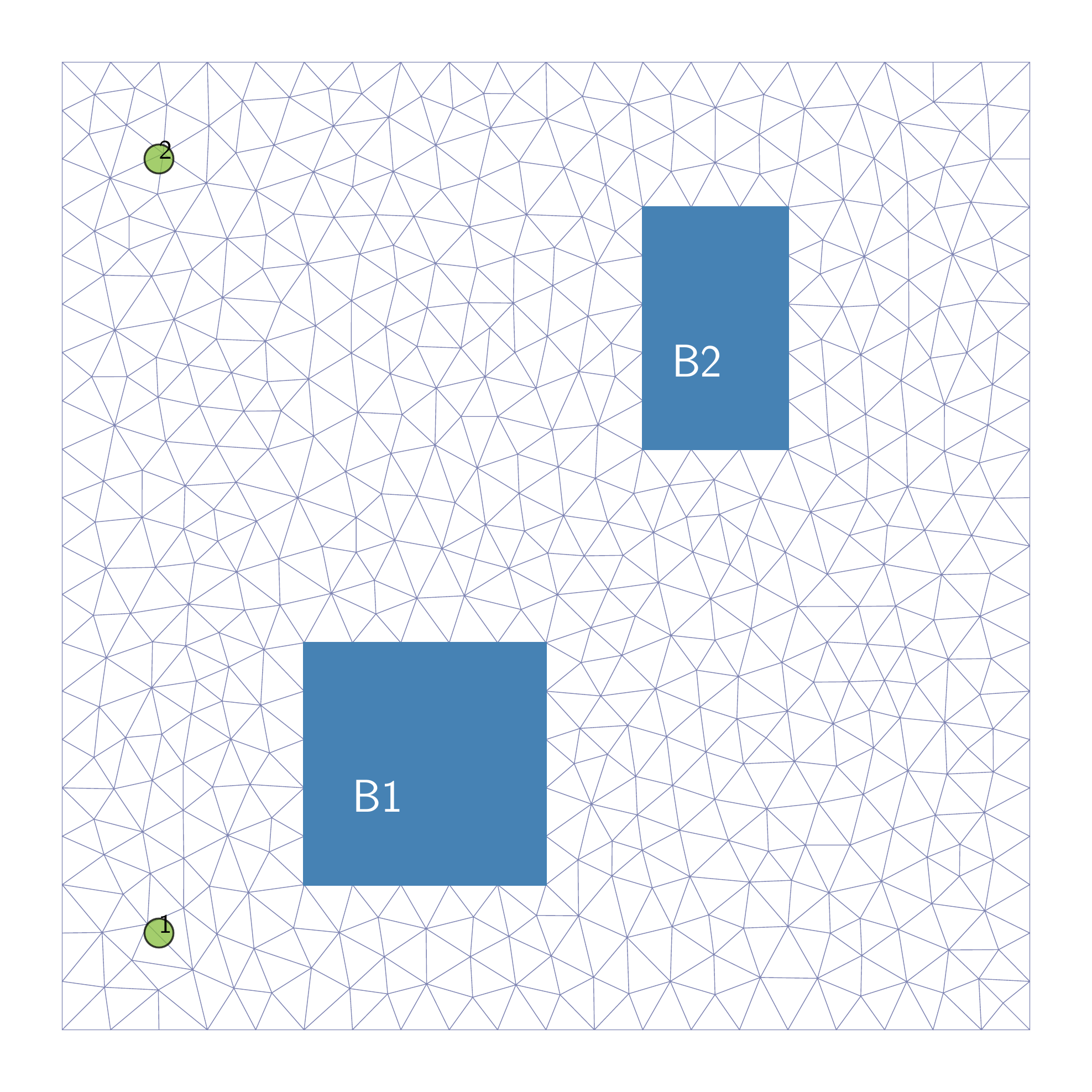}
      \quad
      \includegraphics[width=0.25\linewidth]{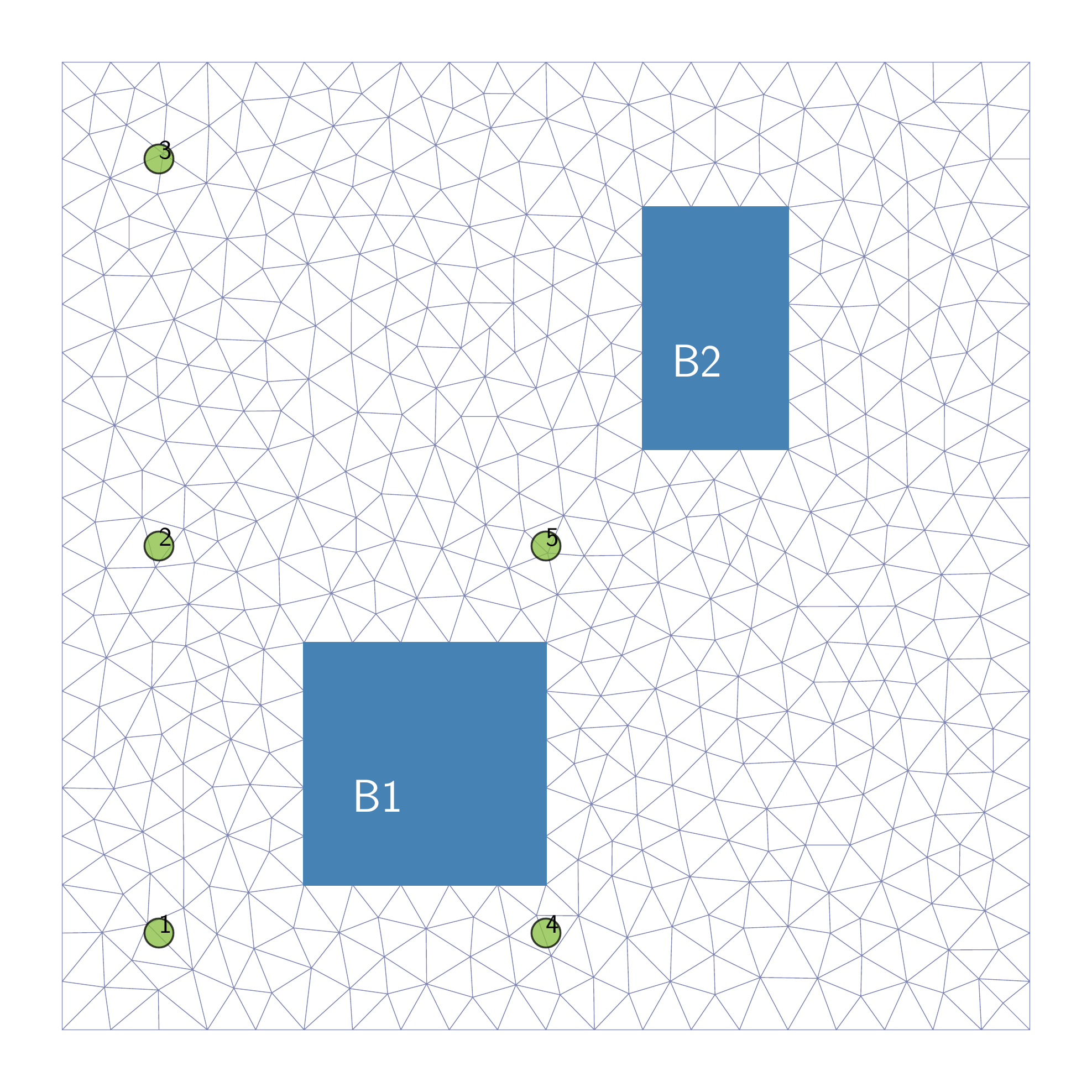}
      \includegraphics[width=0.25\linewidth]{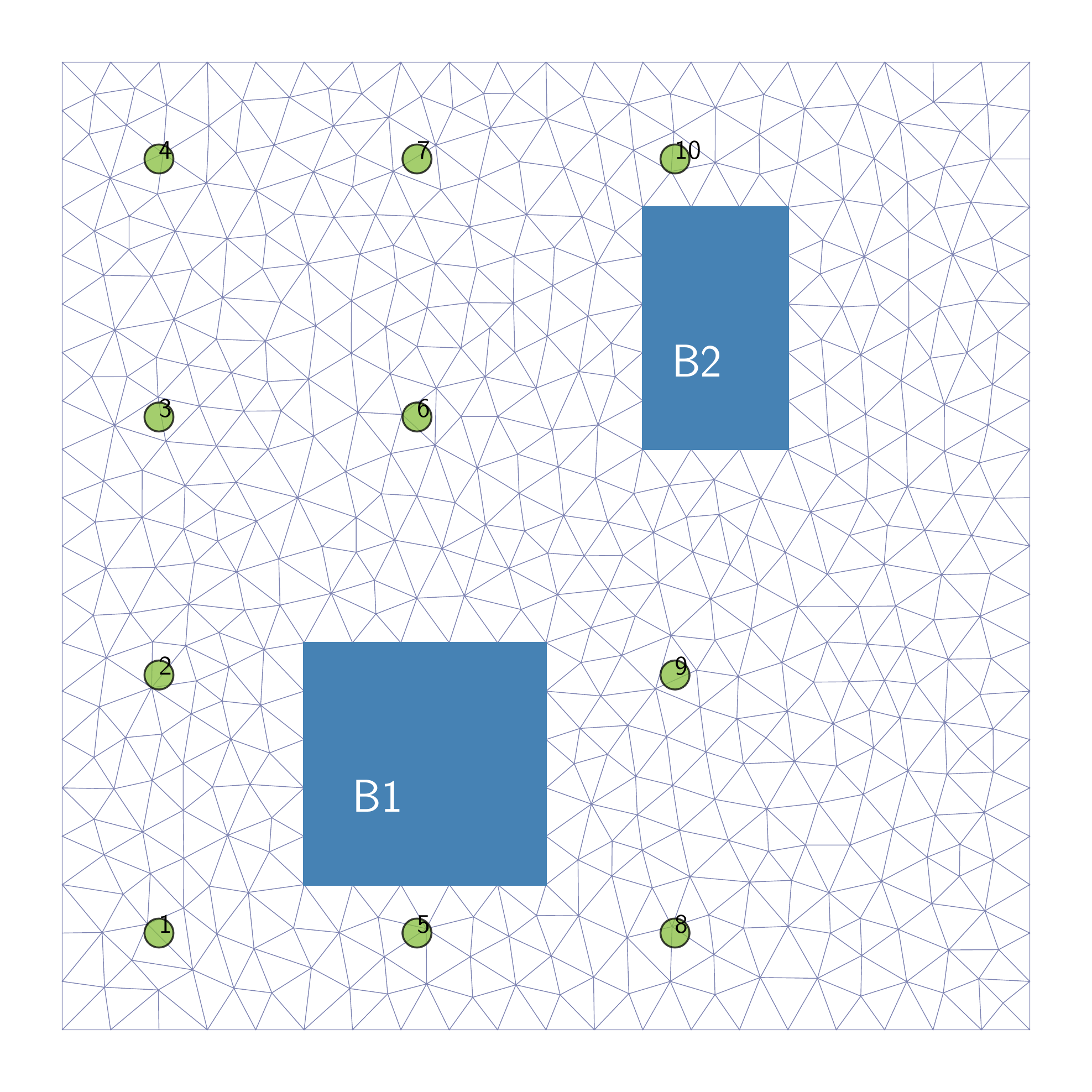}
        \caption{Varying number of candidate sensor locations $\Nsens$ distributed in the domain. 
            Candidate sensors are labeled from $1$ to $\Nsens$ in each case. 
        }
      \label{fig:AD_Observatin_Grids} 
    \end{figure}
    The sensors capture the concentration of the contaminant at the sensor locations at 
    a predefined set of one or more observation time instances 
    $\{ t_1,\, \ldots,\, t_{\nobstimes} \} \subset [0, T]$.
    The number of observation time instances is $\nobstimes\!=\!16$ and is given by 
    $t_1\!+\! s \Delta t$, with initial observation made at
    $t_1\!=\!1$;  $\Delta t\!=\!0.2$ is the simulation timestep of the model;
    and $s=0, 1, \ldots, 20$. 
    Thus, an observation vector $\obs\in\Rnum^{\Nobs}$ represents the measured spatiotemporal 
    concentrations, that is, observations at the predefined locations and time instances 
    with the observation space dimension $\Nobs \!=\! \Nsens \!\times\! \nobstimes$. 

    The observation error distribution is $\GM{\vec{0}}{\Cobsnoise}$, with
    $\Cobsnoise\!\in\! \Rnum^{\Nobs\times\Nobs}$ describing spatiotemporal
    correlations of observational errors, with the design matrix set to the block diagonal matrix 
    $\designmat = \mat{I} \otimes \diag{\design}$ to cope with the fact that spatial observations are stacked to form 
    a spatiotemporal observation vector. Here $\mat{I}\in{\nobstimes \times \nobstimes}$ is an identity matrix, and 
    $\otimes$ is the matrix Kronecker product.

  \subsection{Forward operator, adjoint operator, and the prior}
    The parameter-to-observable map (forward operator) $\F$ is linear and  
    maps the model initial condition (model parameter) $\iparam$ to the observation space.
    Specifically, $\F$ represents a forward simulation over the time interval $[0, T]$ 
    followed by applying a restriction to the sensor locations at the observation times.
    %
    The model adjoint is given by $\F\adj \!:= \!\mat{M}\inv\mat{F}\tran$, where
    $\mat{M}$ is the finite-element mass matrix.
    %
    Following~\cite{bui2013computational}, we use a Gaussian prior $\GM{\iparb}{\Cparampriormat}$
    where the prior covariance matrix $\Cparampriormat$ is a discretization of $\mathcal{A}^{-2}$,
    with a Laplacian $\mathcal{A}$.

  \subsection{The uncertain parameter}
  \label{subsec:uncertain_parameters_setup}
    We assume $\Cobsnoise$ is not exactly known and is parameterized using an uncertain parameter 
    $\uncertainparam$.
    Specifically, in the numerical experiments, we define the observation noise 
    covariance matrix as $\Cobsnoise:\diag{\uncertainparam^{2}}$, where the square is 
    evaluated elementwise  and each entry of 
    $\uncertainparam$ can assume any value in the interval $[0.02,\, 0.04]$.
    Thus, the derivative of the noise matrix required for 
    evaluating~\eqref{eqn:utility_A_opt_gradient_obs_noise_relaxed} 
    takes the form 
    $
    \del{\Cobsnoise(\uncertainparam)}{\uncertainparam_i}=2\uncertainparam_i\vec{e}_i\vec{e}_i\tran 
    $
    over the same domain.
    The bounds of the uncertain parameters are modeled by asserting the following bound constraints on 
    the uncertain parameter in the inner optimization problem:
    \begin{equation}
        0.02 \leq \uncertainparam_i \leq 0.04\,,\quad \forall i=1,\ldots,\Nsens\,. 
    \end{equation}
    %

  \subsection{Optimization setup}
  \label{subsec:optimization_setup}
    As suggested in~\cite{attia2022stochastic}, \Cref{alg:Polyak_MaxMin_Stochastic} is initialized 
    with initial Bernoulli parameters (initial policy)
    $
    \hyperparam^{(0)}\in(0, 1)^{\Nsens} 
    = \left(\hyperparam_1^{(0)},\ldots, \hyperparam_{\Nsens}^{(0)} \right)\tran 
    = \left(0.5,\ldots, 0.5 \right)\tran 
    $; that is, all sensors have equal probability to be active or inactive.
    The maximum number of iterations $\ell$ for all experiments is set to $100$,
    and 
    the algorithm terminates if no improvement
    is achieved in either the inner or the outer optimization problems 
    or if the maximum number of iterations is reached. 
    Lack of improvement in the optimization routines (inner and outer) is realized by 
    setting $\textsc{tol}=1e\!-\!8$ and $\textsc{pgtol}=1e\!-\!8$, where
    the optimizer terminates if either of the following conditions is satisfied: 
      $|\stochobj^{n} - \stochobj^{n-1}|< \textsc{tol}$, where $\stochobj^{(n)}$ is the value of the stochastic 
      objective in the outer problem at the $n$th iteration, or
      $\max_{i = 1, ..., n} | g_i | \leq \textsc{pgtol}$, where $g_i$ is the $i$th 
      component of the projected gradient of the respective optimization problems. 
    The learning rate of the outer stochastic optimization problem 
    (Step~\Cref{algstep:Polyak_MaxMin_Stochastic_hyperparam_opt}) is 
    $\gamma_1=1e\!-\!4$ with the maximum number of iterations 
    set to $5$; see Step~\ref{algstep:Polyak_MaxMin_Stochastic_hyperparm_update}.
    A combination of a small learning rate and a small number of iterations is chosen to prevent 
    the stochastic optimization routine from converging quickly
    to a degenerate policy, which can limit the ability of the algorithm to 
    explore the uncertain parameter space before convergence.

    The inner optimization problem (Step~\Cref{algstep:Polyak_MaxMin_Stochastic_param_opt}) is solved by 
    using a limited-memory Broyden-Fletcher-Goldfarb-Shanno (LBFGS) routine~\cite{zhu1997algorithm} to 
    enforce the bound constraints defined by the uncertain parameter domain described 
    in~\Cref{subsec:uncertain_parameters_setup}, 
    that is, $\uncertainparam_i \in [0.02, 0.04]\, \forall\, i=1,\ldots,\Nsens$. 
    The step size $\gamma_2$ is optimized by using a standard line-search approach.
    In our implementation we use a \pyoed optimization routine that employs an LBFGS implementation provided by the Python scientific 
    package \textsc{SciPy}~\cite{2020SciPy-NMeth}.

    After the algorithm terminates, it generates a sample of $m=5$ designs sampled from the final policy.
    If the final policy $\hyperparam\opt$ is degenerate, all sampled designs are identical and are equal 
    to the policy $\hyperparam\opt$. This is typically achieved if the optimal solution is unique.
    Otherwise, the designs are generated as samples from a multivariate Bernoulli 
    distribution with success probability $\hyperparam\opt$.
    The algorithm picks the design associated with the highest value of the objective.
    However, we also inspect the whole sample for verification in our experiments. 
    
    The stochastic gradient~\eqref{eqn:stoch_generalized_gradient_baseline} is evaluated by setting 
    ensemble sizes to $\Nens={\rm N_b}=32$. 
    While in our experiments we noticed similar performance with a much smaller number of ensembles,
    we use these settings for statistical consistency.

  \subsection{Numerical results}
  \label{subsubsec:numerical_results}
    Here we show numerical results for the experiments described above.
    %
    
    \subsubsection{Results with 2 candidate sensors}
      We show results for $\Nsens=2$ candidate sensors using~\Cref{alg:Polyak_MaxMin_Stochastic}.
      The penalty term is discarded here; this is equivalent to setting the penalty parameter to 
        $\regpenalty=0$ in~\eqref{eqn:robust_stochastic_A_optimal_OED_optimization}.
      \Cref{fig:AD_NumSensors_2_Surface} shows the progress and results obtained by 
      applying~\Cref{alg:Polyak_MaxMin_Stochastic}.
      Specifically, \Cref{fig:AD_NumSensors_2_Surface} (left) illustrates the behavior 
      (over consecutive iterations) of the optimization algorithm compared with expectation surfaces 
      corresponding to all values of the uncertainty parameter $\uncertainparam$.
      \Cref{fig:AD_NumSensors_2_Surface} (right) enables comparing the results of the optimization procedure 
      with respect to a brute-force search.
      Since we cannot enumerate all possible values of the uncertain parameter 
      (unless they form a small set), we carry out the brute-force search over the sampled 
      set $\uncertainparamsample$.
      \begin{figure}[htbp!]
        \center
          \includegraphics[width=0.48\textwidth]{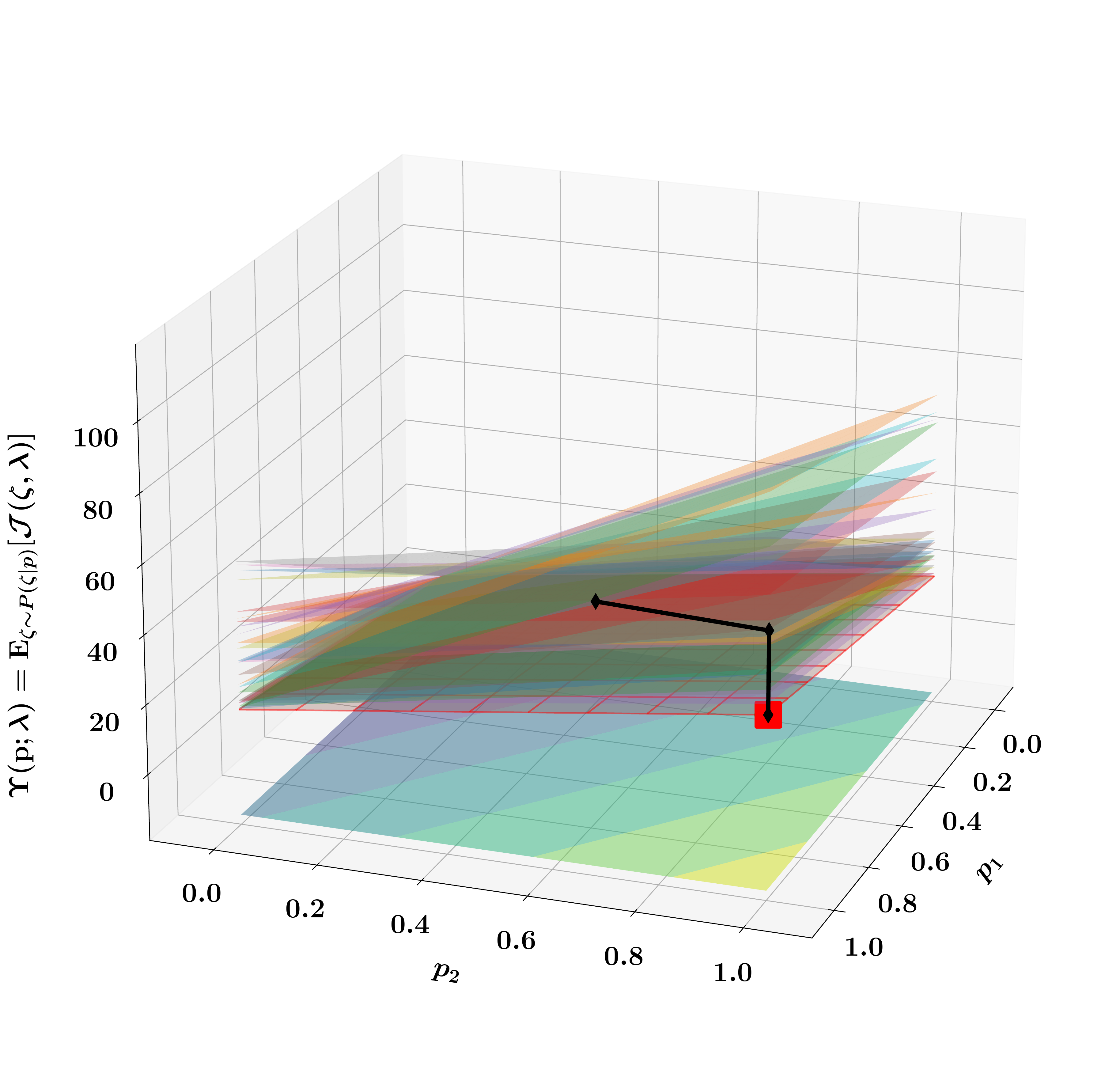}
          \includegraphics[width=0.48\textwidth]{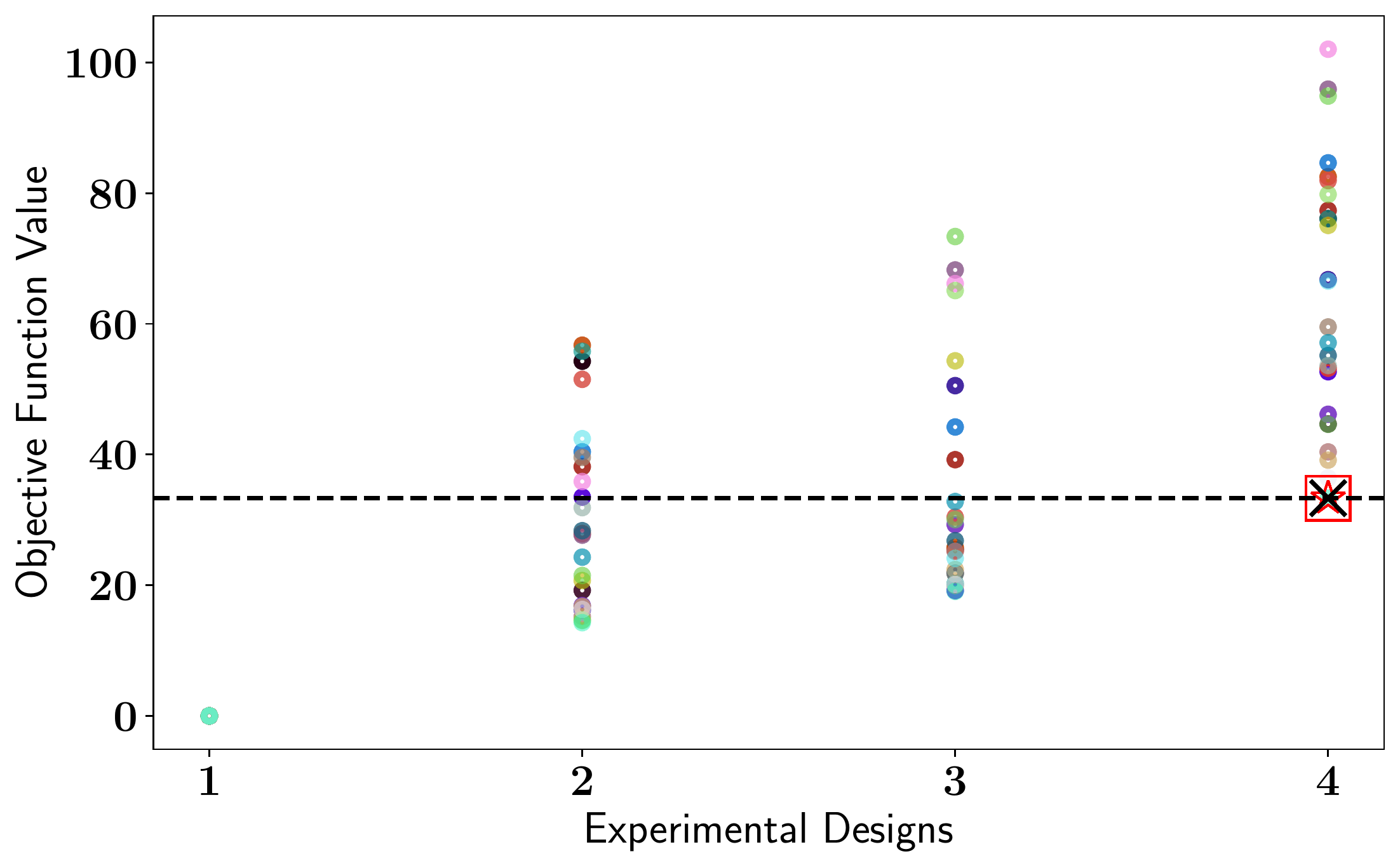}
        \caption{
            Left: surfaces of the stochastic objective function $\stochobj(\design)$
            for multiple values of the uncertainty parameter
            $\uncertainparam=(\sigma_1, \sigma_2)$
            (generated while solving the inner minimization problem in~\Cref{alg:Polyak_MaxMin_Stochastic}).  
            Results (optimization path) of~\Cref{alg:Polyak_MaxMin_Stochastic} are
            also highlighted. 
            In each of the two directions $10$ equally spaced points (values of $\hyperparam_i$) 
            are taken  to create the surface plots.
            Right: the value of the objective $\obj(\design, \uncertainparam)$ evaluated at all possible 
            designs $\design\in\{0, 1\}^{\Nsens=2}$ over the finite sample $\uncertainparamsample$.
            Results are shown for the $4$ possible designs indexed from $1$ to $4$ using the enumeration 
            scheme used in~\eqref{eqn:stochastic_objective}.
            For each value of the uncertain parameter $\uncertainparam$ we select a unique color 
            and plot the values of $\obj$ at all $4$ possible designs as circles of the same color. 
            The optimal solution generated by~\Cref{alg:Polyak_MaxMin_Stochastic} is marked as a red star.
            The black dashed line is plotted at the level of the optimal objective value and thus enables 
            identifying the quality (e.g., optimality gap) of the solution returned by the optimization algorithm.
            Here, no regularization (e.g., sparsification) is enforced; that is, we set the penalty parameter 
            to $\regpenalty=0$.
        }
      \label{fig:AD_NumSensors_2_Surface}
      \end{figure}

      As shown in \Cref{fig:AD_NumSensors_2_Surface} (right), the global optimum solution 
      $\hyperparam\opt=(1, 1)\tran$ is unique, and the 
      algorithm successfully recovers the global optimal solution in this case after $1$ iteration.
      Note that at each iteration of the algorithm the policy parameter $\hyperparam$ is updated, 
      and then the uncertain parameter $\uncertainparam$ is updated, which explains the result 
      in~\Cref{fig:AD_NumSensors_2_Surface} (left). 
      The optimal solution, that is, 
      the optimal policy $\hyperparam\opt$, in this case is degenerate, with all $5$ design samples returned 
      by the algorithm being identical (all equal to the optimal policy $\design=(1, 1)\tran$) as indicated 
      by the red star in \Cref{fig:AD_NumSensors_2_Surface} (right).
      Note that the expectation surfaces formed by the different values of the uncertain parameter intersect,
      as shown in~\Cref{fig:AD_NumSensors_2_Surface} (left). This means that the lower bound is not formed 
      by one specific value of the uncertain parameter, and hence solving the OED optimization problem 
      using a nominal value is likely to result in erroneous results as discussed earlier 
      in~\Cref{subsec:robust_binary_OED_insight}.

    \subsubsection{Results with 5 candidate sensors}
      For $\Nsens=5$ the number of possible observational configurations 
      (that is, the number of possible binary designs) is $2^{5}=32$. 
      Thus, we can obtain the optimal solution by performing a brute-force search over these possible designs,
      in combination with the sampled set $\uncertainparamsample$ of the uncertain parameter, and use it 
      to validate the output of the proposed approach.
      In this experiment we enforce sparsity by setting the penalty term using a sparsity-promoting function.
      Specifically, we use the $\ell_0$ norm to enforce sparsity on the robust design, and we define the penalty function as
      $\penaltyfunction{\design}:= \wnorm{\design}{0}^{2}$. 
      Note that since the design space is binary, $\ell_0$ and $\ell_1$ are equivalent and are both nonsmooth. 
      We set the penalty parameter to $\regpenalty=10$ to achieve a significant level of sparsity. 
      The penalty  parameter $\regpenalty$  generally can be selected to achieve a desired level of sparsity 
      by using heuristic approaches such as the L-curve, which is well developed in the regularization context.
      Figure~\ref{fig:AD_NumSensors_5_L1_Weight_-10} shows the results of~\Cref{alg:Polyak_MaxMin_Stochastic}
      for this experiment.  
      \begin{figure}[htbp!]
      \center
        \includegraphics[width=0.48\textwidth]{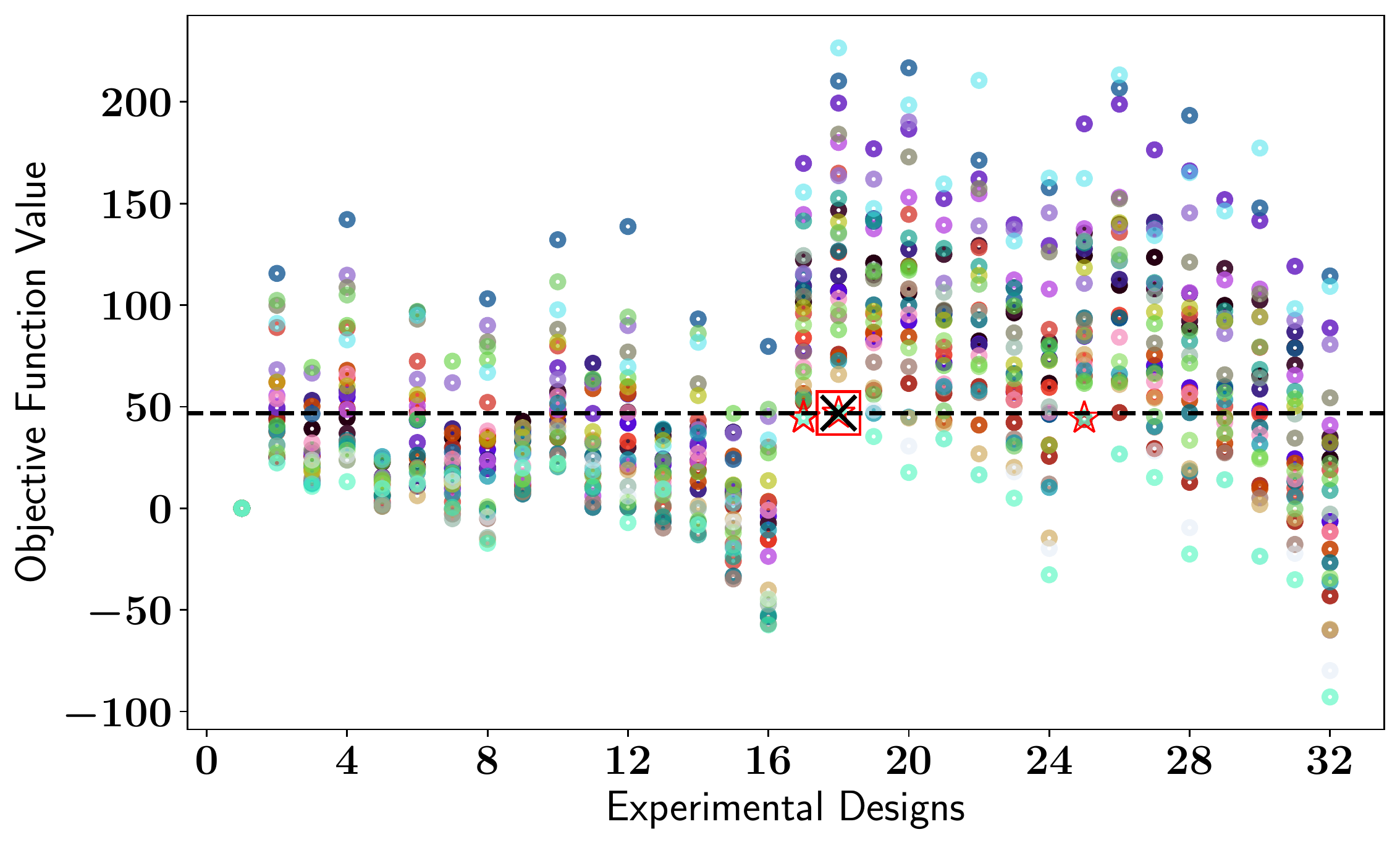}
        \includegraphics[width=0.48\textwidth]{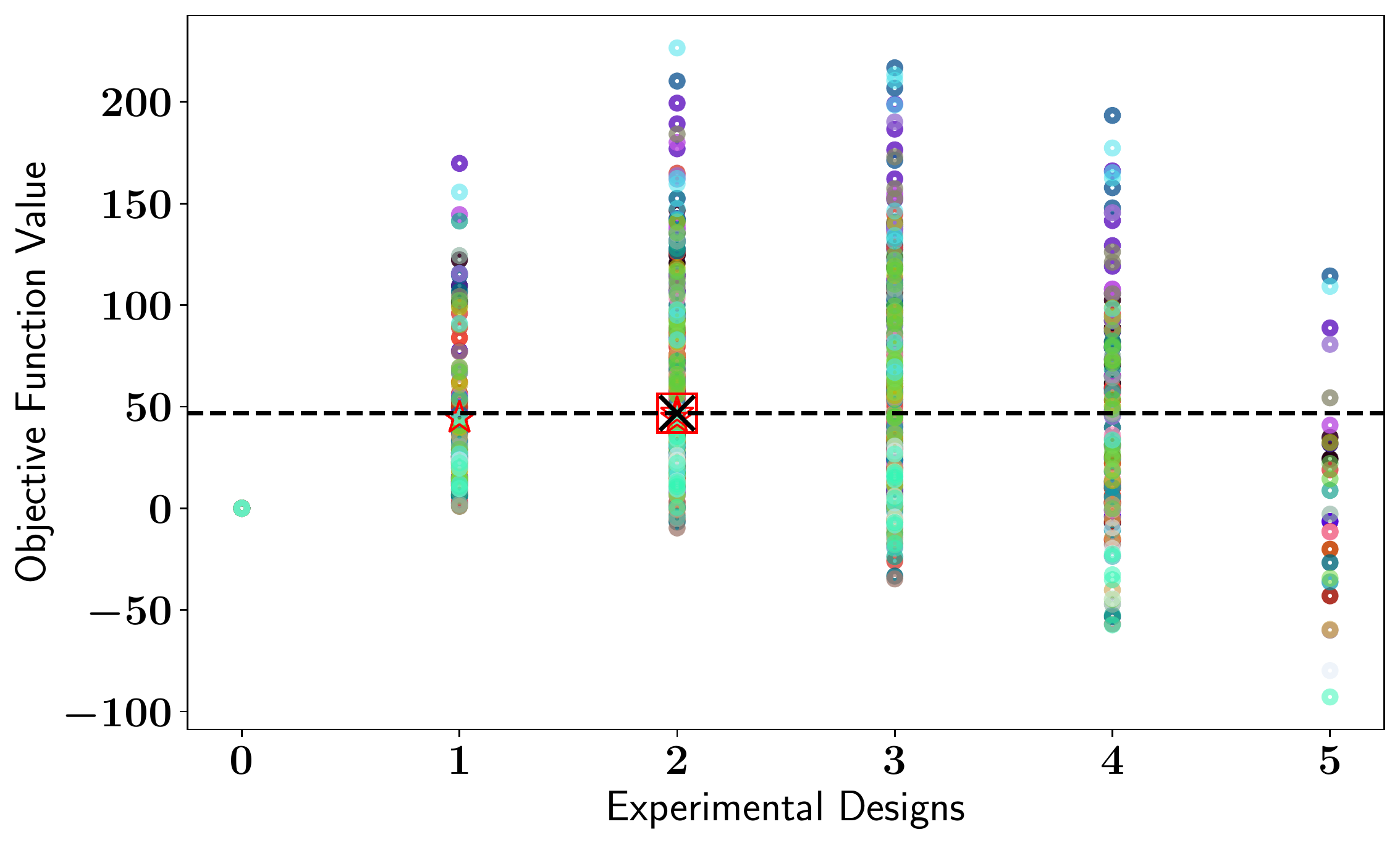}
          \caption{The number of possible designs is $2^{\Nsens}=2^5=32$, each of which is plotted on the $x$-axis.
            For each design, the value of the objective is plotted for all values of the uncertain parameter 
            in the sample generated by the optimization procedure, that is, $\uncertainparam\in\uncertainparamsample$.
            Each value of $\uncertainparam$ has a distinctive color, and the (recommended) optimal solution is marked 
            as a red star.
            The global optimum value (black dashed line) is obtained by finding the max-min value over all possible 
            designs and all values of the uncertain parameter $\uncertainparam\in\uncertainparamsample$ in the 
            sample generated by the optimization algorithm.
            Left: The designs are ordered on the x-axis using the enumeration scheme used in~\eqref{eqn:stochastic_objective}.
            Right: Same as the results in the left panel, but the designs are grouped on the x-axis based on the number of 
            nonzero entries, ranging from $0$ to $\Nsens=5$.
          }
        \label{fig:AD_NumSensors_5_L1_Weight_-10}
      \end{figure}

      \Cref{fig:AD_NumSensors_5_L1_Weight_-10} (left) shows the objective values evaluated at the sample of designs 
      (here $5$ designs are sampled from the final policy) returned by the optimization algorithm after it was terminated.
      The algorithm is terminated here at the maximum number of iterations ($100$), with the final policy being 
      nondegenerate; that is, some entries of the Bernoulli activation probabilities $\hyperparam$ are nonbinary, 
      giving us a chance to explore multiple high-quality designs near the optimal solution. 
      The sampled designs either cover or are very close the global optimum. 
      Thus the algorithm can point to multiple candidate solutions at or near the global optimum 
      and give practitioners or field experts a final decision if needed. 
       Here, however, the solution recommended by the algorithm (the design associated with largest objective value among the sampled designs) 
      in fact matches the global optimum solution.
      The level of sparsity in the designs sampled from the final policy can be better inspected by 
      reconfiguring the brute-force plot, namely,~\Cref{fig:AD_NumSensors_5_L1_Weight_-10} (left), 
      such that the designs on the $x-$axis are grouped by the number of active sensors, that is, the number of entries 
      in $\design$ equal to $1$ as shown in~\Cref{fig:AD_NumSensors_5_L1_Weight_-10} (right). 
      
      The sensor locations suggested by the algorithm here, compared with the global robust optimal solution, are shown in 
      \Cref{fig:AD_NumSensors_5_L1_Weight_-10_Designs}.
      \begin{figure}[htbp!]
      \center
        \includegraphics[width=0.24\textwidth]{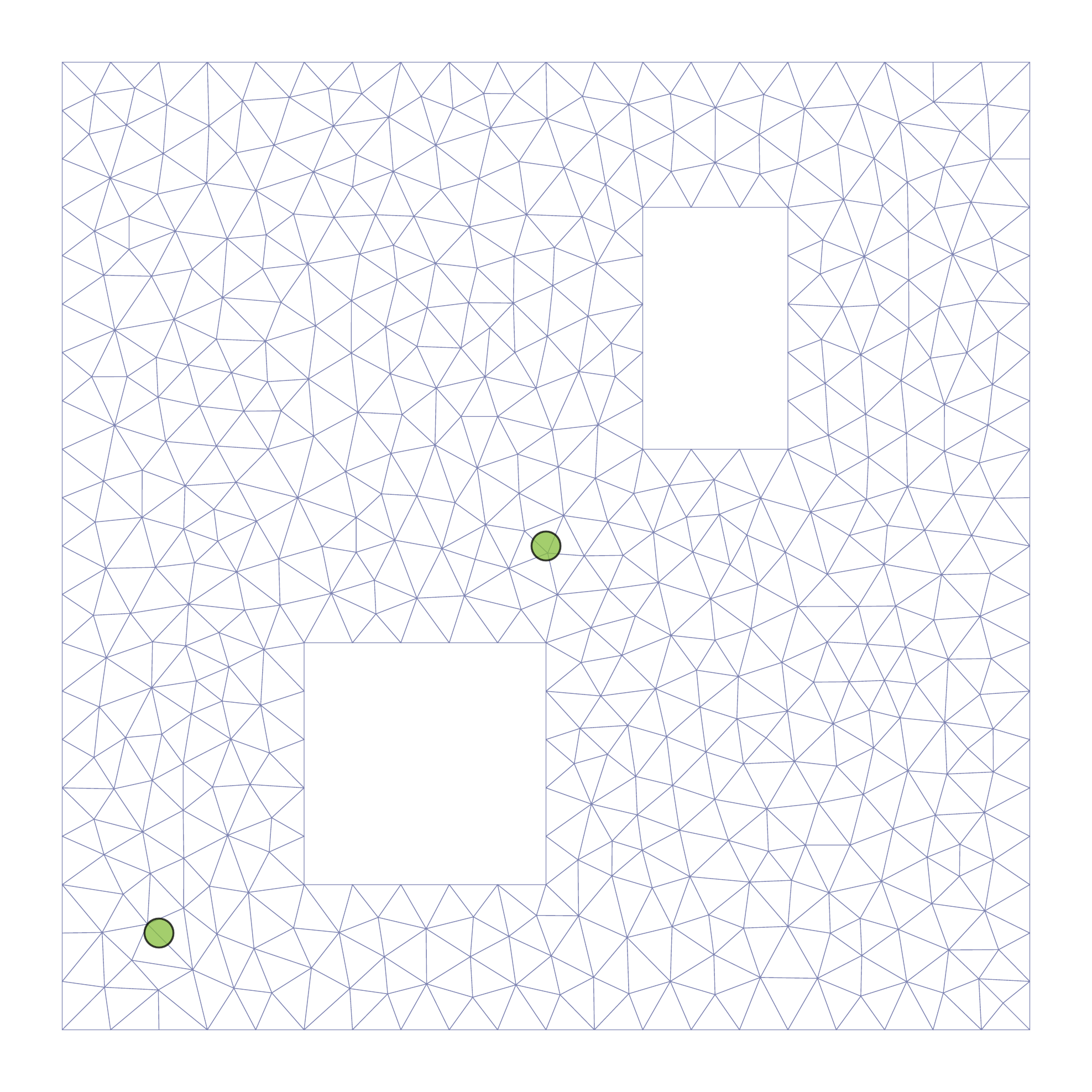}
        \includegraphics[width=0.24\textwidth]{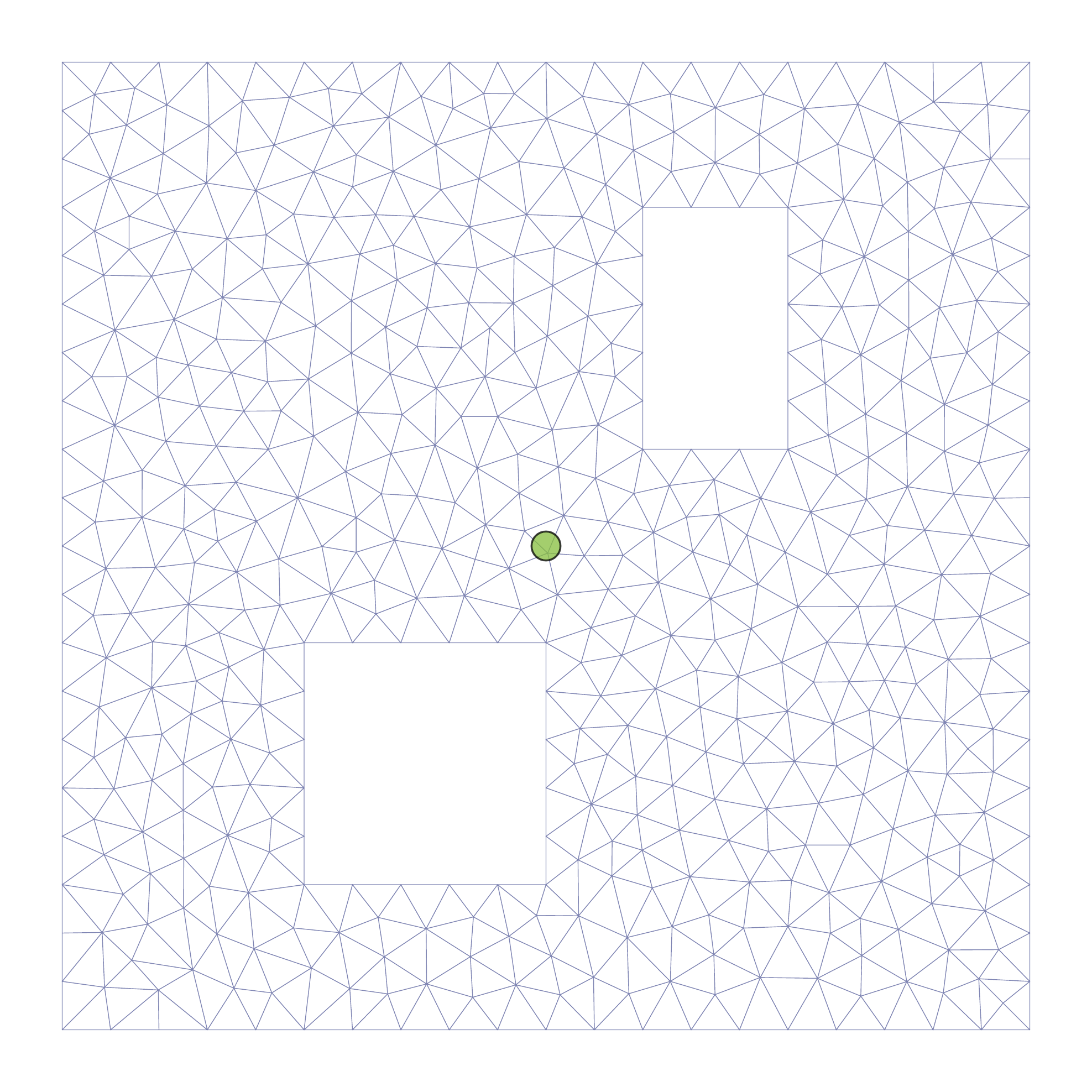}
        \includegraphics[width=0.24\textwidth]{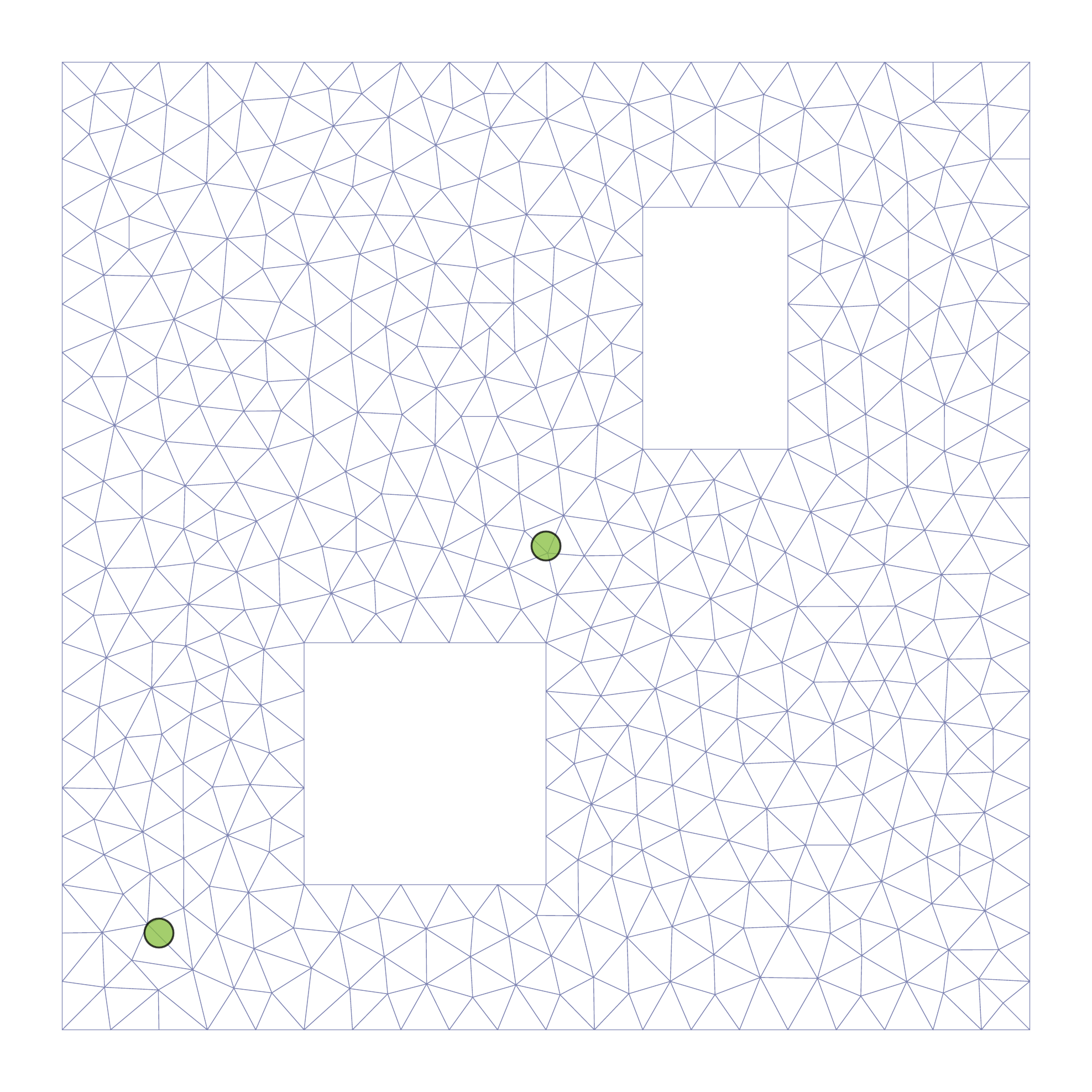}
        \includegraphics[width=0.24\textwidth]{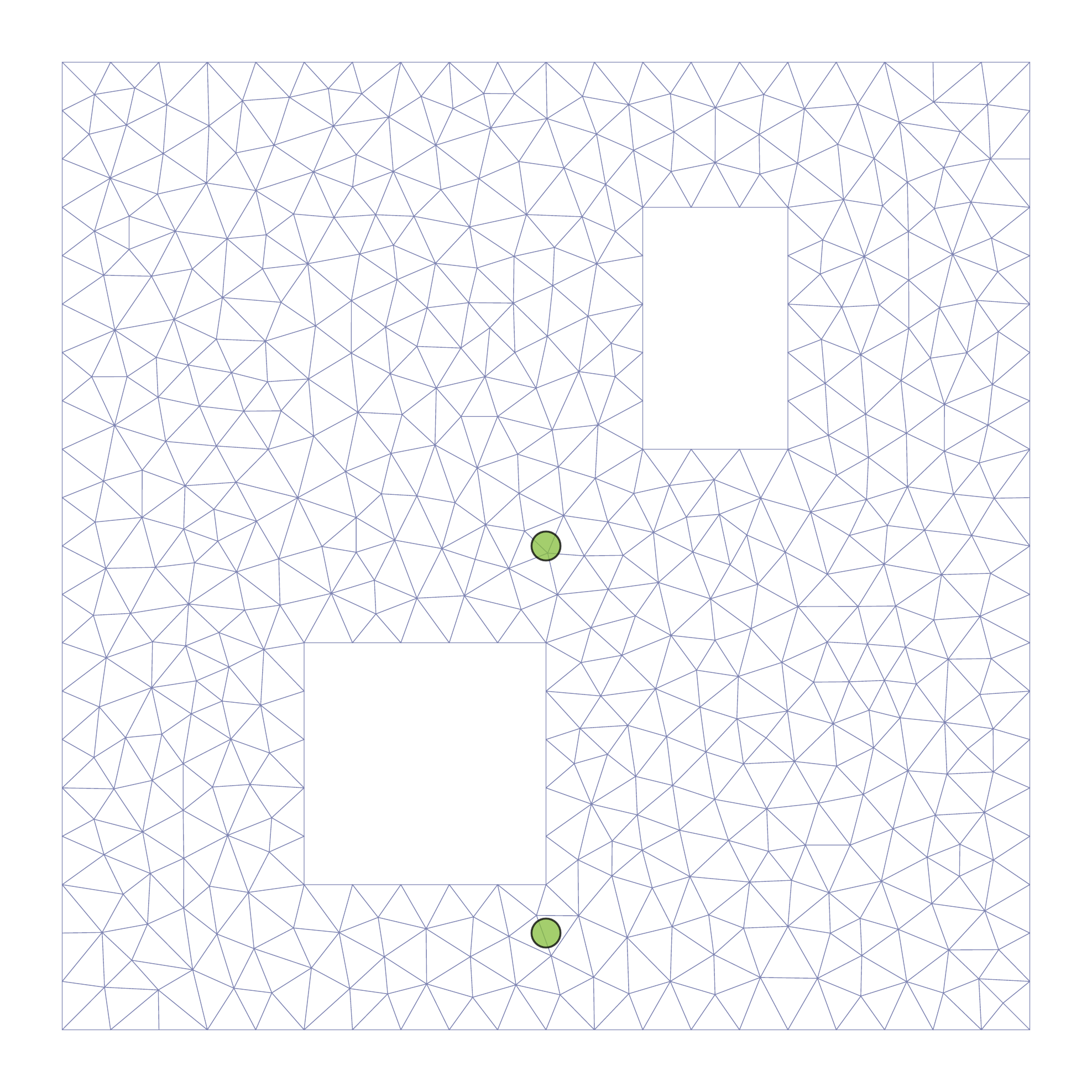}
          \caption{Results corresponding to those shown in~\Cref{fig:AD_NumSensors_5_L1_Weight_-10}.
            Here we show the global optimum design (left), compared with the unique designs ($3$ rightmost) in the sample 
            generated by the algorithm.
          }
        \label{fig:AD_NumSensors_5_L1_Weight_-10_Designs}
      \end{figure}

      \Cref{fig:AD_NumSensors_5_L1_Weight_-10_Iterations} shows the performance of \Cref{alg:Polyak_MaxMin_Stochastic} 
      over consecutive iterations. 
      Here we cannot visualize surfaces of the objective function values as we did in the two-dimensional 
      case discussed above. 
      Nevertheless, we show the value of the stochastic objective $\stochobj(\design, \uncertainparam)$ evaluated at 
      samples from the policy generated at the consecutive iterations of the optimization algorithm; see 
      \Cref{fig:AD_NumSensors_5_L1_Weight_-10_Iterations} (left).
      Note that at each iteration of the algorithm, the stochastic objective value is evaluated and plotted twice: 
      first when the policy is updated by the outer optimization routine and second when the uncertain parameter 
      is updated in the inner optimization routine. 
      \begin{figure}[htbp!]
      \center
        \includegraphics[width=0.48\textwidth]{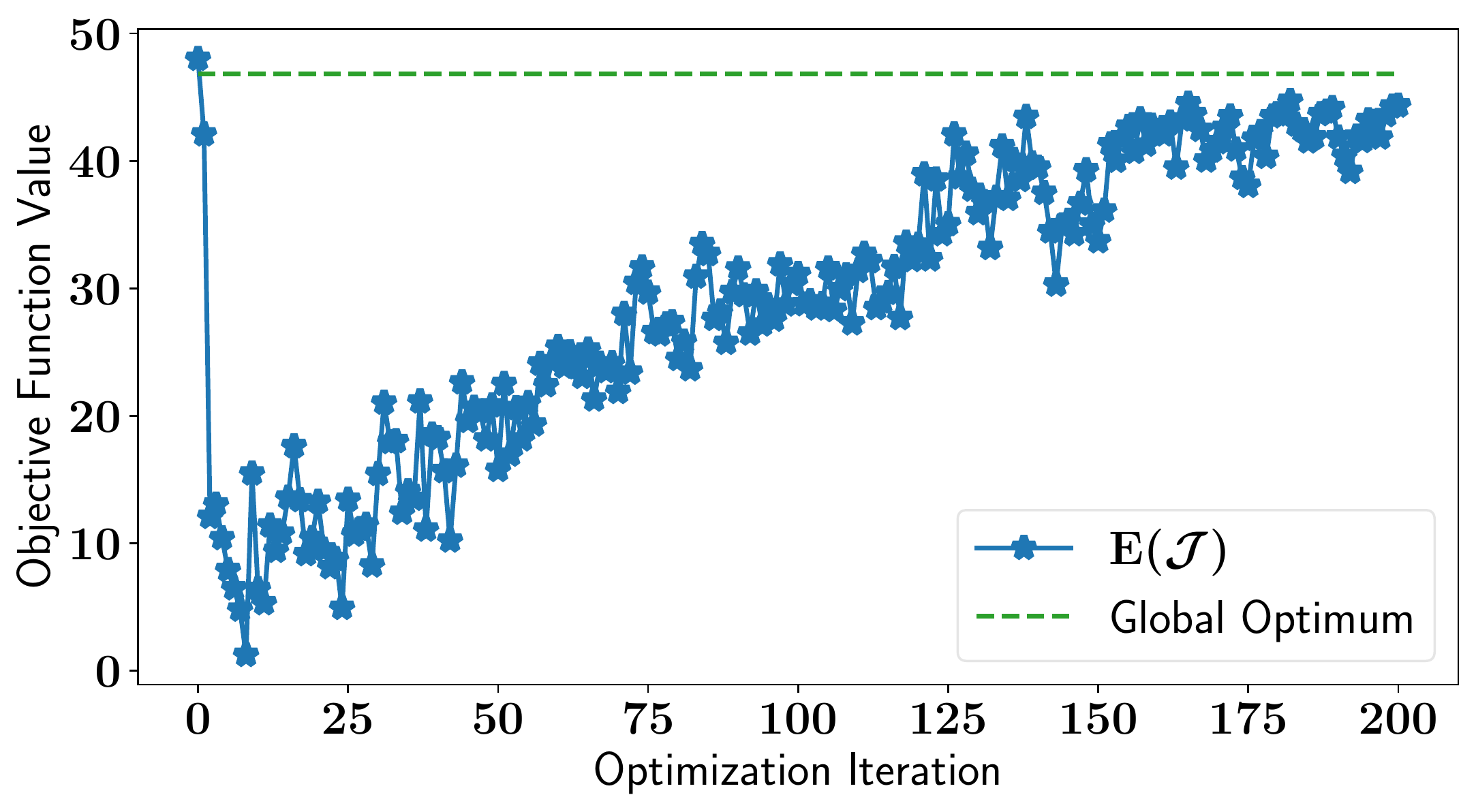}
        \includegraphics[width=0.48\textwidth]{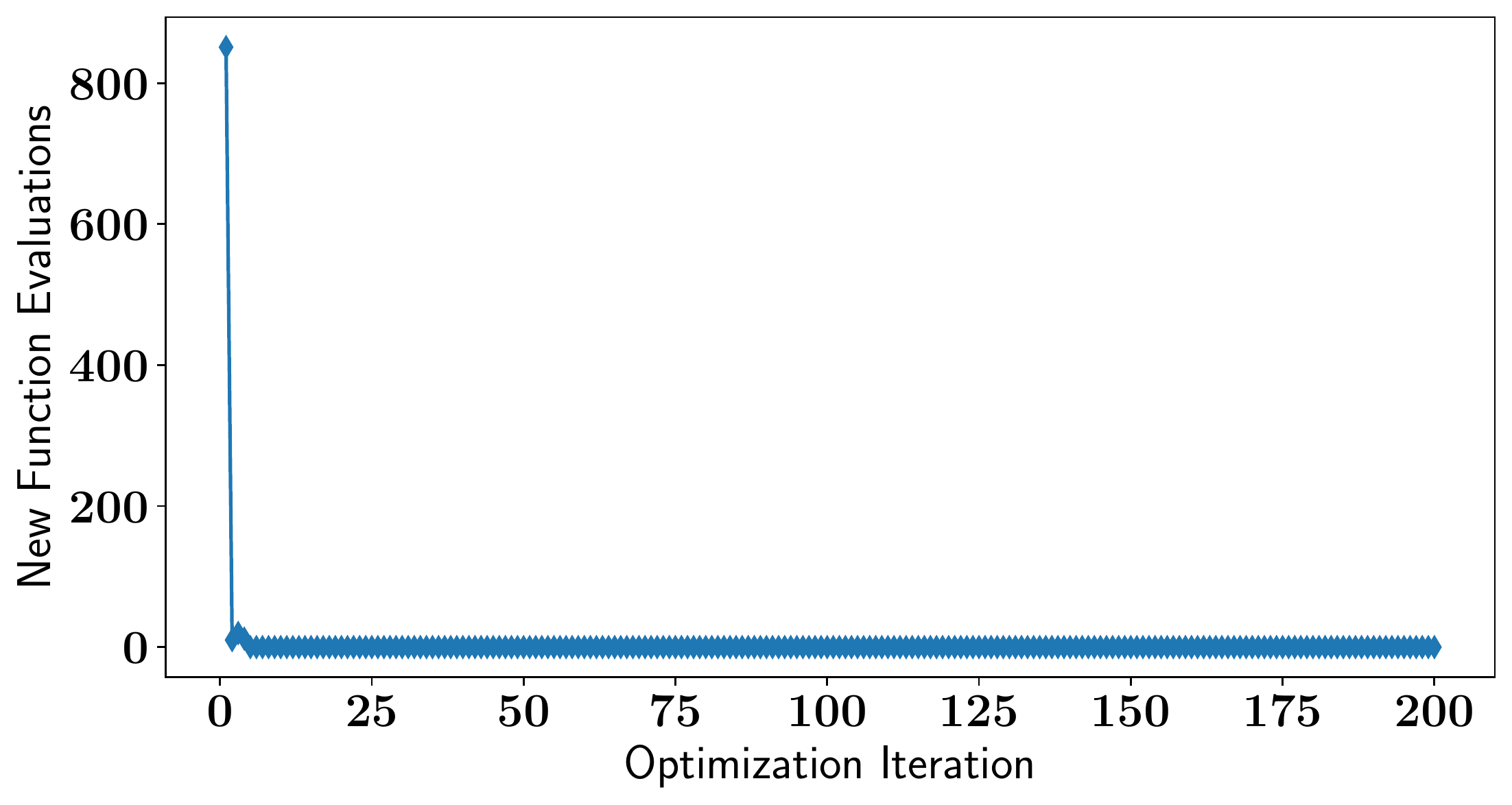}
          \caption{Results corresponding to those shown in~\Cref{fig:AD_NumSensors_5_L1_Weight_-10}.
            Left: value of the stochastic objective  $\stochobj(\design, \uncertainparam)$ evaluated at samples 
              from the policy generated at the consecutive iterations of the optimization algorithm.
            Right:  number of new evaluations of the objective $\obj$ incurred at each iteration of the algorithm,
              that is, the number of function evaluations at pairs $(\design, \uncertainparam)$ that the algorithm 
              has not inspected previously.
            \label{fig:AD_NumSensors_5_L1_Weight_-10_Iterations}
        }
      \end{figure}

      An important aspect of the stochastic learning approach employed by~\Cref{alg:Polyak_MaxMin_Stochastic} is 
      the fact that the algorithm tends to reutilize design samples it has previously visited as it moves toward a 
      degenerate policy, that is, an optimal solution, that presents a notable computational advantage of the algorithm.
      This is supported by~\Cref{fig:AD_NumSensors_5_L1_Weight_-10_Iterations} (right), which shows the number of 
      new function evaluations carried out at each iteration. 
      This is further discussed in~\Cref{subsubsec:efficiency} 
      where we describe an approach to keep track of $(\design, \uncertainparam)$ pairs that the algorithm inspects 
      that can be reutilized instead of reevaluating the objective at the recorded pairs.

      \commentout{
          A nominal value test can highlight the importance of accounting for misspecification and uncertainties while solving the 
          OED problem.
          \Cref{fig:AD_NumSensors_5_Nominal} shows the results of a nominal value test carried out with the present setup.
          The results show that a randomly selected nominal value of the uncertain parameter is not conservative in general.
          Specifically, \Cref{fig:AD_NumSensors_5_Nominal} shows that the nominal value is better on a small number of points (orange), 
          which we suspect is near the nominal value. The robust solution, however, is better on more parameter values (larger blue area).
          \sven{Please update the comment here.} 
          \begin{figure}[htbp!]
          \center
            \includegraphics[width=0.45\textwidth]{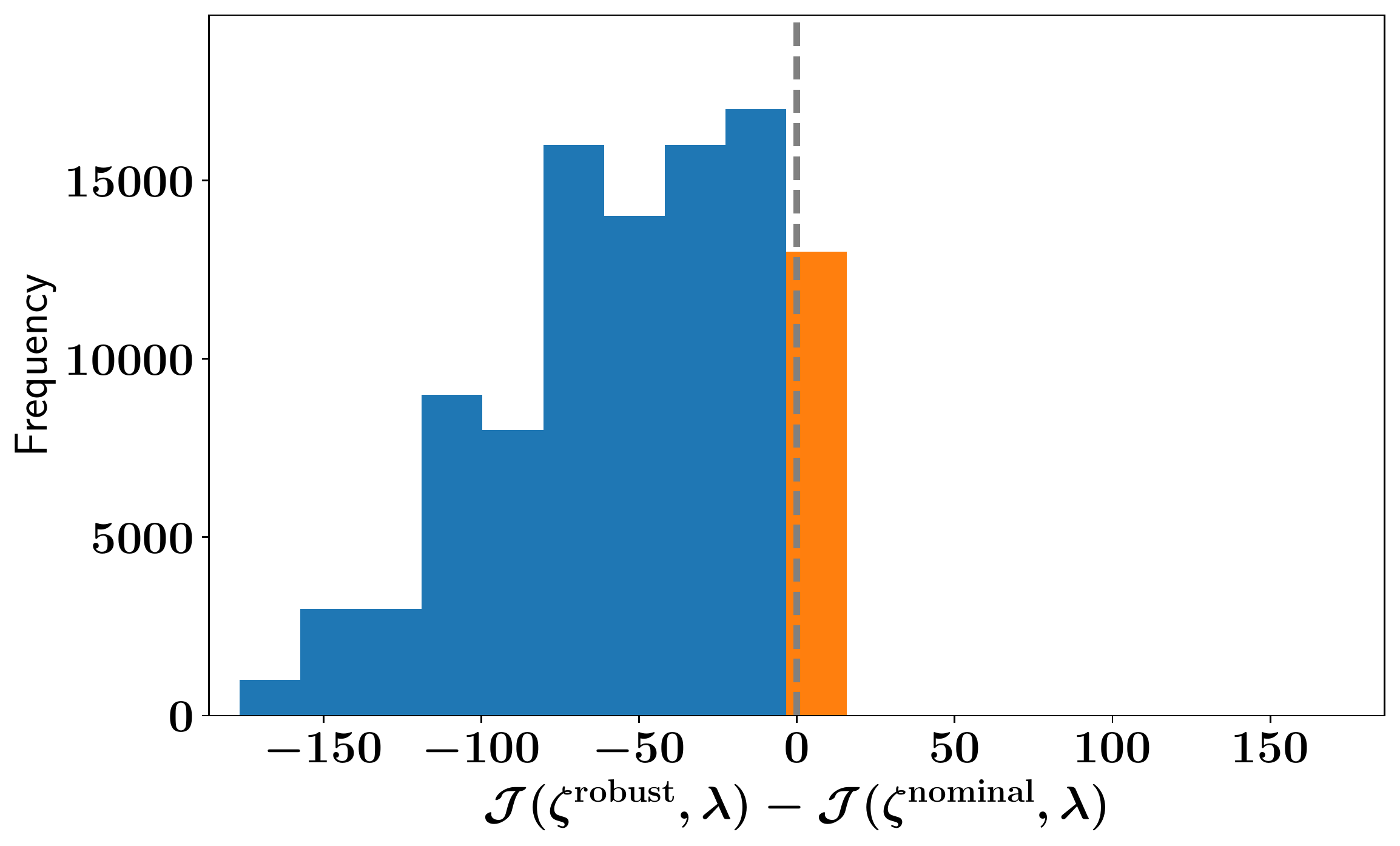}
              \caption{
                  Nominal value test corresponding to results in~\Cref{fig:AD_NumSensors_5_L1_Weight_-10}.
                  The robust solution $\design^{\rm robust}$ is obtained by solving the robust OED optimization 
                  problem~\eqref{eqn:robust_stochastic_binary_optimization}.
                  The solution $\design^{\rm nominal}$ is obtained by solving the stochastic (non-robust)
                  OED problem~\eqref{eqn:stochastic_OED_optimization} for a nominal value of the uncertain parameter 
                  $\uncertainparam=(0.03, 0.03, 0.03, 0.03, 0.03)\tran$.
                  The x-asis shows values of $\obj(\design^{\rm robust}, \uncertainparam) - \obj(\design^{\rm nominal}, \uncertainparam)$ 
                  evaluated at $25$ uniformly distributed points in each entry of the uncertain parameter $\uncertainparam$, 
                  resulting in $25^5=9765625$ sample points.
            }
            \label{fig:AD_NumSensors_5_Nominal}
          \end{figure}
        }

    \subsubsection{Results with 10 candidate sensors}
      A common case is to be limited to a budget $\budget$ of sensors when resources are limited.
      In this case one can enforce the budget constraint by setting the penalty term, for example, to 
      $\penaltyfunction{\design}:=\abs{ \wnorm{\design}{0} - \budget}$ and enforce this constraint by choosing 
      the value of the penalty parameter $\regpenalty$ large enough. 
      Here we set the number of candidate sensors to $\Nsens=10$ and the budget to $\budget=3$.
      In this case the number of possible design choices is $2^{10}$ with only $\ncr{10}{3}=120$ designs satisfying 
          the budget constraint. 
      Each of these possible designs can be evaluated by using any value of the uncertain parameter $\uncertainparam$.
      Results of this experiment are summarized in~\Cref{fig:AD_NumSensors_10_L0_Weight_-50_Budget_3_Bruteforce} 
      and show  behavior similar to that in the previous cases. 
      The algorithm returns a sample of designs with objective values almost identical to the global optimum 
      with similar sparsity level.
      %
      \begin{figure}[htbp!]
      \center
        \includegraphics[width=0.48\textwidth]{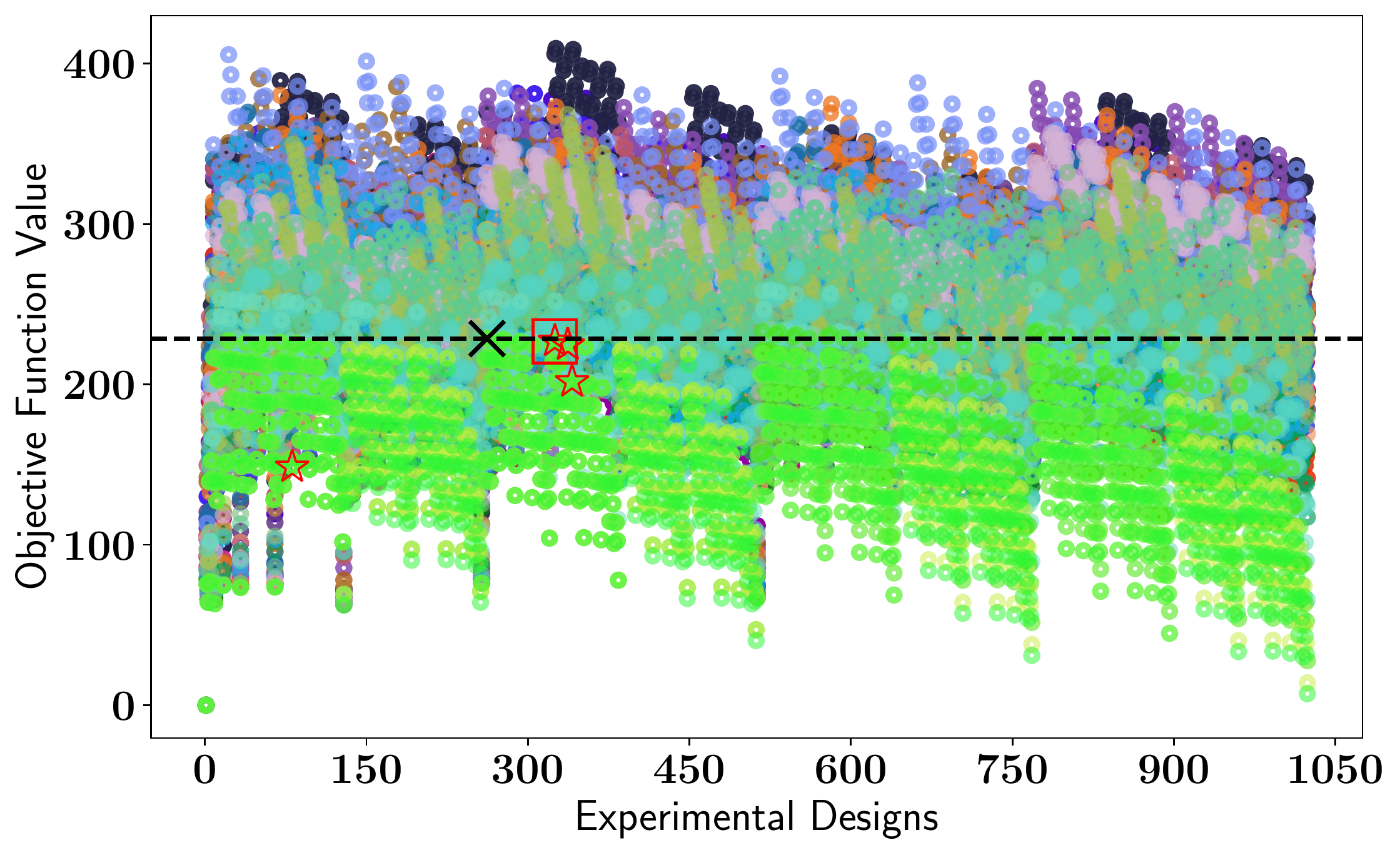}
        \includegraphics[width=0.48\textwidth]{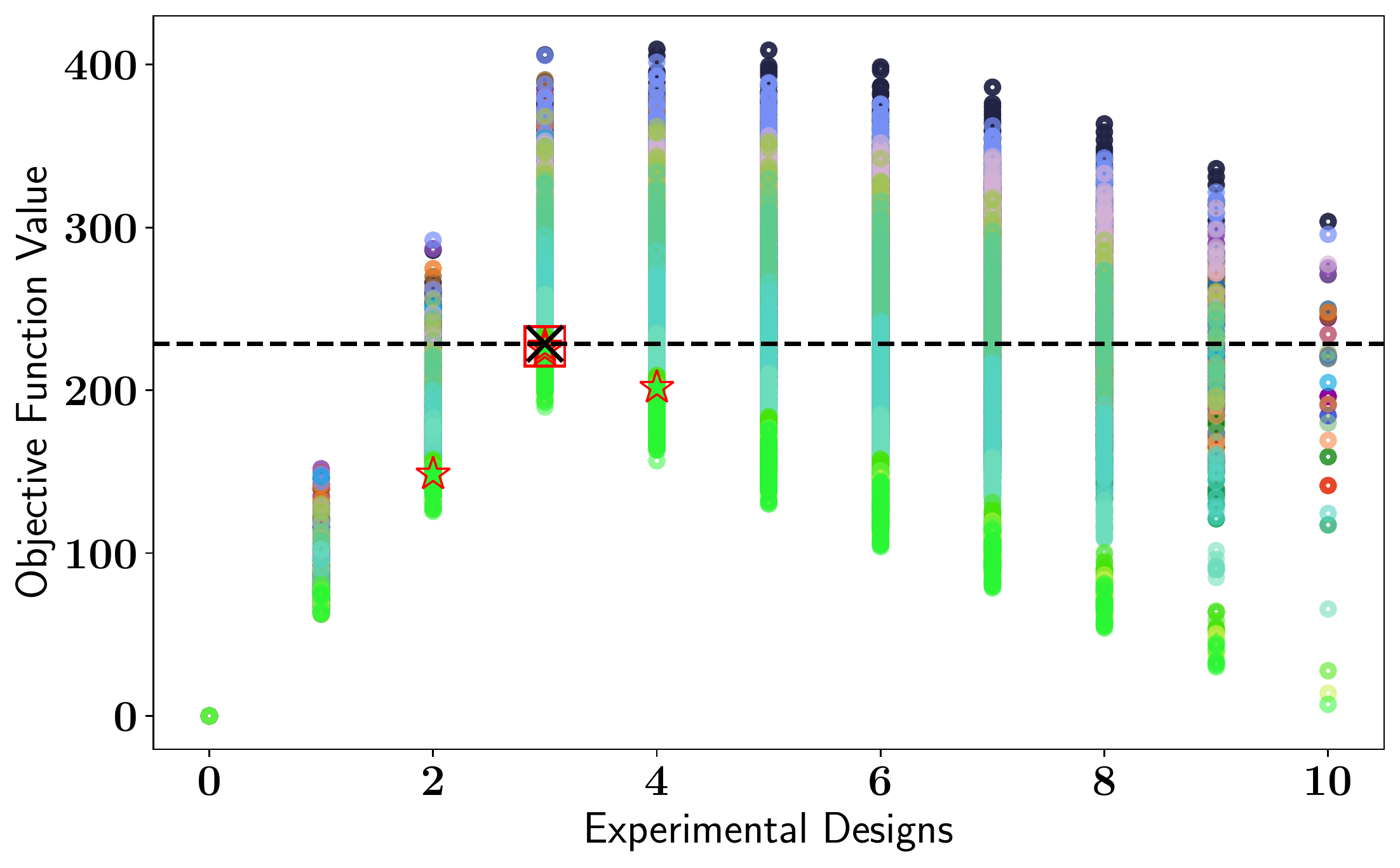}
          \caption{Similar to~\Cref{fig:AD_NumSensors_5_L1_Weight_-10}. 
          Here the number of candidate sensors is set to $\Nsens=10$, and a budget of $\budget=3$ active sensors
          is enforced by setting the penalty function to $\penaltyfunction{\design}:=\abs{ \wnorm{\design}{0} - 3}$ and 
          the penalty parameter  to $\regpenalty=50$.
          Left: the designs are indexed on the x-axis by using the enumeration scheme~\eqref{eqn:stochastic_objective}. 
          Right: the designs are grouped on the x-axis by the number of active sensors, that is, the number 
          of entries in $\design$ equal to $1$.
          }
        \label{fig:AD_NumSensors_10_L0_Weight_-50_Budget_3_Bruteforce}
      \end{figure}
      %

    \subsubsection{Efficiency and reutilization of sampled designs}
    \label{subsubsec:efficiency}
      \Cref{alg:Polyak_MaxMin_Stochastic} follows the approach presented in~\cite{attia2022stochastic} for solving the outer optimization problem over 
      the finite binary feasible domain $\{0, 1\}^{\Nsens}$. 
      As the algorithm proceeds toward an optimal solution, the activation probability $\hyperparam$ moves toward parameters of a degenerate Bernoulli distribution,
      that is, the activation probabilities move toward $\{0, 1\}$ for all entries of the design $\design$.
      The variance of a Bernoulli random variable is equal to $\hyperparam(1-\hyperparam)$, which means the variability is at its minimum of $0$ for 
      values of the activation probability closer to the bounds $\{0, 1\}$ and the variability is at its highest value at $\hyperparam=0.5$.
      Thus, as the algorithm proceeds toward the optimization variable, that is, the activation probability $\hyperparam$ moves to a degenerate distribution,
      the algorithm tends to resample designs  it has already utilized in previous iterations, thus yielding significant reduction in computational cost. 
      This was leveraged in~\cite{attia2022stochastic} by keeping track of the indexes of the sampled designs 
      $\design$ and the corresponding objective value $\obj(\design)$, for example in a dictionary. 
      At each iteration, when the value of the objective is required, this dictionary is inspected first, 
      and the value of the objective is calculated only if it is not available in this dictionary.

      A similar strategy is followed in this work and is utilized in~\Cref{alg:Polyak_MaxMin_Stochastic}. 
      Here we need to keep track of both the sampled designs $\design$ and the values of the uncertain 
      parameter $\uncertainparam$ along with the corresponding value of the objective $\obj(\design, \uncertainparam)$. 
      We use the following simple hashing approach. 
      The design and the uncertain parameter are augmented in a vector $\vec{h}=[\design\tran, \uncertainparam\tran]$,
      and a unique hash value is then created.
      %
      This hash value is set as the key in a dictionary, with the value equal to the objective $\obj(\design, \uncertainparam)$.
      This procedure is invoked each time the value of the objective $\obj$ is requested, thus enabling us to look 
      up previously calculated results first.
      This approach turns out to be very effective because it removes redundant calculations, leading to a considerable savings 
      in computational cost, especially for the outer optimization problem. 
      This is evident by results in~\Cref{fig:scalability_AD} that show the behavior of the proposed algorithm 
      over consecutive iterations for increasing problem cardinality $\Nsens$.
      Specifically,~\Cref{fig:scalability_AD} (left) shows the number of new objective function calls
      (that is, the number of evaluations of the objective $\obj$ for new combinations of $\design$ and $\uncertainparam$) 
      at consecutive iterations for increasing cardinality. For clarity, we truncate the figure at $8$ iterations 
      and note that for all remaining iterations the number of new functions is $0$ for all cardinalities $\Nsens$ in this
      experiment.
      \begin{figure}[htbp!]
      \center
        \includegraphics[width=0.49\textwidth]{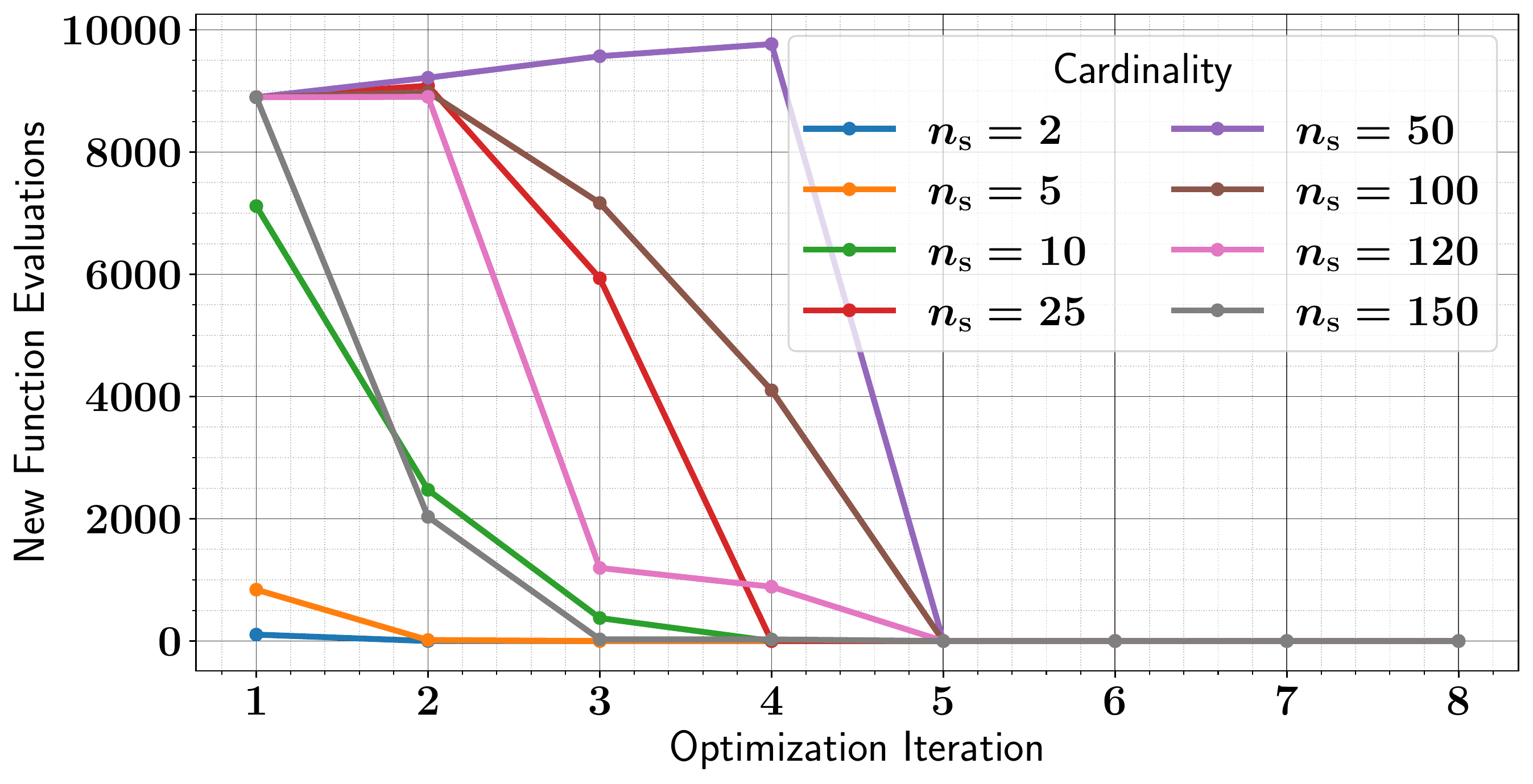}
          \hfill
        \includegraphics[width=0.49\textwidth]{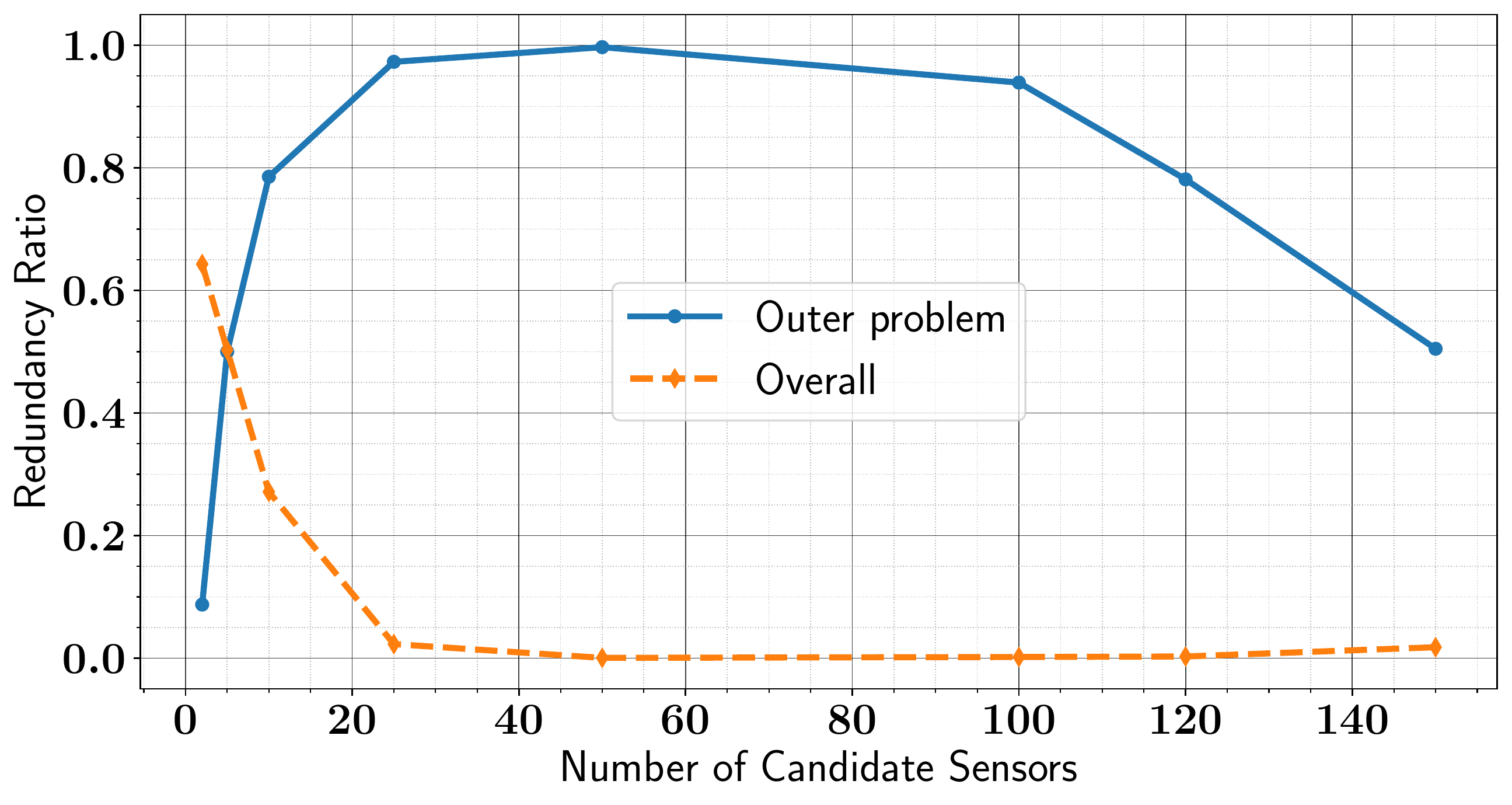}
          \caption{Results of~\Cref{alg:Polyak_MaxMin_Stochastic} for increasing cardinality $\Nsens$.
          Left: Number of new function evaluations $\obj(\design, \uncertainparam)$ incurred at each iteration.
          Right: Total redundancy ratio for the outer optimization problem and the overall redundancy ratio.
        }
        \label{fig:scalability_AD}
      \end{figure}
      \Cref{fig:scalability_AD} (right) shows the redundancy ratio defined as
      $1-\frac{\text{Number of unique pairs }(\design, \uncertainparam)}{\text{Total number of
        sampled designs}(\design, \uncertainparam)}$.
      The redundancy ratio is evaluated for the overall optimization procedure (both outer and inner optimization problems combined) 
      and is also shown for the outer optimization problem.
      These results are consistent with~\cite{attia2022stochastic} as the stochastic learning procedure used to update the Bernoulli 
      hyperparameter $\hyperparam$ tends to reutilize previously sampled designs as it moves closer to the optimal solution.
      The overall ratio here is dominated by evaluations of the objective in the line search 
      employed by the inner optimization routine; see~\Cref{subsec:optimization_setup}.
      However, the results in the two panels in~\Cref{fig:scalability_AD} together explain the efficiency of the optimization procedure 
      and show that the algorithm tends to navigate quickly to plausible (near-optimal) combinations of designs and uncertain 
      parameter values, in a very small number of steps and function evaluations, and then reutilizes those values until the 
      algorithm converges or terminates. 

      \commentout{
          \ahmed{I may add one more set of results indicating RMES and posterior FIM for the global, 
            optimal, and random designs with random samples of the uncertain parameter. Results are being generated.
          }
      }

\section{Discussion and Concluding Remarks}
\label{sec:conclusions}
      In this work we proposed a new approach for robust binary optimization problems 
      with application to OED for sensor placement.
      The approach accurately formulates the robust optimization problem as a stochastic program 
      and solves it following an efficient sampling-based stochastic optimization approach. 

      The approach provided in this work inherits both advantages and limitations of the stochastic 
      binary optimization approach~\cite{attia2022stochastic}.
      Three main challenges are associated with this work.
      First is the optimal choice of the learning rate of the outer optimization problem. 
      This, in fact, is a common hurdle prevalent in most stochastic optimization routines.
      An empirical approach is typically followed to choose the learning rate and is followed here. 
      Specifically, a set of learning rates are used to solve the optimization problem, and the rate 
      that results in favorable behavior is used. 
      However, such an approach is not practical for large-dimensional problems, 
      and automatic tuning of such parameters becomes crucial. This will be considered in future work.
      Second, the proposed solution algorithm can generate designs with objective values close to the 
      global optimum. It can, however, be entrapped in a local optimum.
      The entrapment in a local optimum is expected when the outer optimization problem converges quickly to 
      a binary design (or a degenerate policy with binary parameter $\hyperparam$) 
      or due to nonconvexity of the expectation surface, that is, the stochastic objective function. 
      For the former case, when the algorithm converges quickly to a degenerate policy (activate probability), 
      in consequent iterations the outer maximization problem does not escape that value 
      (the degenerate policy), and the inner optimization problem updates the value of the uncertain parameter only 
      at that corner point of the domain.
      This can be ameliorated by taking only a small number of steps (or a smaller stepsize/learning rate) at 
      each iteration to stay in the interior of the domain while updating the sample of the uncertain parameter. 
      In the latter case, because of the nonuniqueness of the global optimum solution or potential nonconvexity of 
      the expectation surface on the interior of the probability domain $(0, 1)$, the algorithm can be entrapped
      at  a suboptimal nondegenerate policy. This is ameliorated by adding random perturbation to the policy
      when the algorithm finds a local optimum.
      %
      Third, while the numerical results indicate good performance of the optimization approach, 
      there are no global optimality guarantees. 
      Global optimality is challenging and will be investigated in separate work, for example, by modeling 
      the activation probabilities themselves as random variables and considering efficient sampling approaches such 
      as the Hamiltonian Monte Carlo, which can be employed for global optimization.

\appendix

\section{Gradient Derivation: Robust A-Optimality}
\label{app:sec:Gradients}
    In this section we provide a detailed derivation of the gradient required for
    implementing~\Cref{alg:Polyak_MaxMin_Stochastic}.
    Specifically,~\Cref{app:subsec:Robust_A-opt_utility_gradient:param} provides a derivation of the gradient
    of the utility function with respect to the uncertain parameter, namely, 
    the parameterized observation error covariance matrix.
    This gradient is summarized by
    \eqref{eqn:utility_A_opt_gradient_obs_noise_relaxed} 
    and is utilized by
    Step~\ref{algstep:Polyak_MaxMin_Stochastic_obs_noise_gradient}
    of~\Cref{alg:Polyak_MaxMin_Stochastic}.
    %
    First we provide, in~\Cref{lemma:pseudoinv_deriv}, a closed form of the pointwise 
    derivative of the design-weighted precision matrix $\wdesignmat(\design, \uncertainparam)$.
    This identity is valid for the general weighting scheme~\eqref{eqn:pointwise_weighted_precision} 
    and thus applies to the special case of binary designs defined by~\eqref{eqn:PrePost_weighted_Precision}.

    \begin{lemma}\label{lemma:pseudoinv_deriv}
        Given the definition~\eqref{eqn:pointwise_weighted_precision} of the weighted precision 
        matrix $\wdesignmat(\design)$, 
        it follows that 
        \begin{equation}\label{eqn:pseudoinv_deriv}
            \partial\wdesignmat(\design, \uncertainparam)
            = - \wdesignmat(\design, \uncertainparam)
                \partial \left(\designmat(\design) \odot \Cobsnoise(\uncertainparam)\right)
            \wdesignmat(\design, \uncertainparam) \,.
        \end{equation}
    \end{lemma}
    \begin{proof}
        We consider two cases. 
        First, let $\design\in(0,  1]^{\Nsens}$. 
        In this case
        $
        \wdesignmat
            = \inverse{\designmat\odot\Cobsnoise}\,, 
        $ and thus, by applying the rule of the derivative of a matrix inverse, it follows that
        $
        \partial\wdesignmat
            = - \wdesignmat
            \partial \wdesignmat\inv 
            \wdesignmat 
            = - \wdesignmat
            \partial \left(\designmat \odot\Cobsnoise \right) 
            \wdesignmat \,.
        $ 

        Second, consider the case where the design $\design$ is at the boundary of the relaxed
        domain $[0, 1]^{\Nsens}$ with some entries of the design $\design$ are equal to zero.
        Following the same argument in~\cite[Appendix A]{attia2022optimal}, 
        let us represent the noise covariance matrix $\Cobsnoise$ and the weighting
        matrix $\designmat$ as block matrices:
        \begin{equation}
          \Cobsnoise = 
            \begin{bmatrix} \mat{A} & \mat{B} \\ \mat{B}\tran & \mat{D}
            \end{bmatrix}\,; \qquad
            \designmat =
            \begin{bmatrix} \mat{W_A} & \mat{W_B} \\ \mat{W}\tran_{\mat{B}} & \mat{W_D}
            \end{bmatrix}\,; \qquad
            \designmat \odot \Cobsnoise
            =
            \begin{bmatrix} 
                \mat{A}\odot\mat{W_A} & \mat{B}\odot\mat{W_B} \\ 
                \mat{B}\tran\odot\mat{W}\tran_{\mat{B}} & \mat{D}\odot\mat{W_D}
            \end{bmatrix}\,,  
        \end{equation}
        where we assume that both $\mat{A},\, \mat{W_A}$ correspond to elements of 
        $\design$ that are equal to $0$; permutation can be used to reach this form. 
        In this case, by definition of the weights~\eqref{eqn:design_weights}, 
        $\mat{W_A}=\mat{0}$, and $\mat{W_B}=\mat{0}$ 
        \begin{equation}\label{eqn:pseudoinv_weight}
        \wdesignmat = \pseudoinverse{\designmat \odot \Cobsnoise}
            = \begin{bmatrix} \mat{0} & \mat{0} \\ \mat{0} & \inverse{\mat{D}\odot\mat{W_D}} 
          \end{bmatrix}\,.
        \end{equation}
        Thus in this case
        \begin{equation}\label{eqn:partial_pseudoinv_weight}
          \partial \wdesignmat
          = \partial\pseudoinverse{\Cobsnoise \odot \designmat} 
            = \begin{bmatrix} \mat{0} & \mat{0} \\ \mat{0} & 
                - \inverse{\mat{D}\odot\mat{W_D}}
                    \partial \left(\mat{D}\odot\mat{W_D}\right) 
                    \inverse{\mat{D}\odot\mat{W_D}}
          \end{bmatrix}\,,
        \end{equation}
        and
        \begin{equation}\label{eqn:partial_weight}          
            \partial\left(\Cobsnoise \odot \designmat\right)
            = \begin{bmatrix} \mat{0} & \mat{0} \\ \mat{0} & 
                \partial \left(\mat{D}\odot\mat{W_D}\right) 
          \end{bmatrix}\,.
        \end{equation}

        Thus, \eqref{eqn:pseudoinv_deriv} follows immediately by combining~\eqref{eqn:pseudoinv_weight} 
        and~\eqref{eqn:partial_weight}: 
        \begin{equation}
            \wdesignmat \partial \left(\designmat\odot\Cobsnoise\right) \wdesignmat
            = 
            \begin{bmatrix} 
                \mat{0} & \mat{0} \\ \mat{0} & \inverse{\mat{D}\odot\mat{W_D}} 
            \end{bmatrix}
            \begin{bmatrix} \mat{0} & \mat{0} \\ \mat{0} & 
                \partial \left(\mat{D}\odot\mat{W_D}\right) 
            \end{bmatrix}
            \begin{bmatrix} 
                \mat{0} & \mat{0} \\ \mat{0} & \inverse{\mat{D}\odot\mat{W_D}} 
            \end{bmatrix}
            = 
            - \partial \wdesignmat
            \,,
        \end{equation}
        which completes the proof.
        \end{proof}

        Note that~\Cref{lemma:pseudoinv_deriv} holds whether the derivative is taken with respect to the 
        relaxed design (through the weight matrix $\designmat$)
        or the uncertain parameter (through the observation noise matrix $\Cobsnoise$).

  \subsection{Derivative with respect to the uncertain parameter}
  \label{app:subsec:Robust_A-opt_utility_gradient:param}
    %
    Here we derive the gradient 
    $
     \nabla_{\uncertainparam} \obj(\design,\uncertainparam)
      = \nabla_{\uncertainparam} \utilityfunc(\design,\uncertainparam)$
    required for numerically solving the optimization problem in
    Step~\ref{algstep:Polyak_MaxMin_Stochastic_param_opt} and Step~\ref{algstep:Polyak_MaxMin_Stochastic_inner_min}
    of~\Cref{alg:Polyak_MaxMin_Stochastic}.
    We utilize the fact that the penalty function $\penaltyfunc$ is
    independent from the uncertainty parameter $\uncertainparam$.
    For the uncertain parameter
    $\uncertainparam=\uncertainparam^{\rm noise}$, 
    \begin{equation}\label{eqn:utility_A_opt_derivative_obs_noise_relaxed_app}
      \begin{aligned}
        \del{ \utilityfunc(\design,\uncertainparam^{\rm noise})}{\uncertainparam^{\rm noise}_j} 
          &= \del{ \Trace{ 
              \F\adj \wdesignmat(\design; \uncertainparam^{\rm noise} ) \F
             +\Cparampriormat\inv
            } 
            }{{\uncertainparam^{\rm noise}_j}} \\
          &\stackrel{\eqref{eqn:pseudoinv_deriv}}{=} 
          - \Trace{  
            \F\adj
              \wdesignmat(\design;\uncertainparam^{\rm noise})
                \left( \designmat(\design) \odot 
                 \del{  \Cobsnoise(\uncertainparam^{\rm noise} ) }{{\uncertainparam^{\rm noise}_j}}  
                  \right)
              \wdesignmat(\design;\uncertainparam^{\rm noise})
                \F
            }
          \,,
      \end{aligned}
    \end{equation}
    which results in the final form of the gradient:
    \begin{equation}\label{eqn:utility_A_opt_gradient_obs_noise_relaxed_app}
        \nabla_{\uncertainparam^{\rm noise}} \utilityfunc(\design,\uncertainparam^{\rm noise}) 
        = - \sum_{j}^{}{
          \Trace{  
            \F\adj
              \wdesignmat(\design;\uncertainparam^{\rm noise})
                \left( \designmat(\design) \odot 
                 \del{  \Cobsnoise(\uncertainparam^{\rm noise} ) }{{\uncertainparam^{\rm noise}_j}}  
                  \right)
              \wdesignmat(\design;\uncertainparam^{\rm noise})
                \F
            } 
            \vec{e}_j
          }
        \,,
    \end{equation}
    where $\vec{e}_j$ the $j$th unit vector in $\Rnum^{n}$, and $n$ is the dimension of $\uncertainparam$.
    %



\bibliographystyle{siamplain}
\bibliography{references}


\iftrue
\null \vfill
  \begin{flushright}
  \scriptsize \framebox{\parbox{5.5in}{
  The submitted manuscript has been created by UChicago Argonne, LLC,
  Operator of Argonne National Laboratory (``Argonne"). Argonne, a
  U.S. Department of Energy Office of Science laboratory, is operated
  under Contract No. DE-AC02-06CH11357. The U.S. Government retains for
  itself, and others acting on its behalf, a paid-up nonexclusive,
  irrevocable worldwide license in said article to reproduce, prepare
  derivative works, distribute copies to the public, and perform
  publicly and display publicly, by or on behalf of the Government.
  The Department of
  Energy will provide public access to these results of federally sponsored research in accordance
  with the DOE Public Access Plan. http://energy.gov/downloads/doe-public-access-plan. }}
  \normalsize
  \end{flushright}
\fi


\end{document}